\documentclass[11pt]{article}
\usepackage{amsmath}
\usepackage{bm}
\usepackage{verbatim}
\usepackage{color}
\usepackage[breaklinks]{hyperref} %,backref
\usepackage {amsfonts, latexsym}
\usepackage[mathscr]{eucal}

\usepackage{latexsym,bm}
\usepackage{mathrsfs,amssymb,amsthm,ascmac}   % Commented by Matthew
\usepackage[dvips]{graphicx}
\usepackage{wrapfig}
\usepackage{mathabx}
\usepackage{extarrows}

\theoremstyle{plain}
\newtheorem{definition}{Definition}

\newtheorem{theorem}{Theorem}
\newtheorem{proposition}{Proposition}
\newtheorem{lemma}{Lemma}
\newtheorem{corollary}{Corollary}

\numberwithin{claim}{lemma}
\newtheorem*{remark}{Remark}

\newcommand{\N}{\mathbb{N}}

\newcommand{\om}{\omega}

\newcommand{\lceilx}{ \ulcorner }
\newcommand{\rceilx}{ \urcorner }
\newcommand{\inter}{\mathbf{INT}}
\newcommand{\cont}{\mathbf{CONT}}
\newcommand{\pre}{\mathbf{PRE}}
\newcommand{\alg}{\mathbf{ALG}}

\newcommand{\subconserv}{\mathrel{\mathchoice%
  {\sqsubseteq\mkern-15mu{\raise0.17em\hbox{\scriptsize$\mathrm{c}$}\mkern5mu}}%
  {\sqsubseteq\mkern-15mu{\raise0.17em\hbox{\scriptsize$\mathrm{c}$}\mkern5mu}}%
  {\sqsubseteq\mkern-11mu{\raise0.12em\hbox{\tiny$\mathrm{c}$}\mkern4mu}}%
  {\sqsubseteq\mkern-11mu{\raise0.12em\hbox{\tiny$\mathrm{c}$}\mkern4mu}}%
}}
\newcommand{\supconserv}{\mathrel{\mathchoice%
  {\sqsupseteq\mkern-15mu{\raise0.17em\hbox{\scriptsize$\mathrm{c}$}\mkern5mu}}%
  {\sqsupseteq\mkern-15mu{\raise0.17em\hbox{\scriptsize$\mathrm{c}$}\mkern5mu}}%
  {\sqsupseteq\mkern-11mu{\raise0.12em\hbox{\tiny$\mathrm{c}$}\mkern4mu}}%
  {\sqsupseteq\mkern-11mu{\raise0.12em\hbox{\tiny$\mathrm{c}$}\mkern4mu}}%
}}
\def\<#1>{\langle #1 \rangle}

\def\I#1{{{\mathrm{\mathbf I}}(#1)}}
\def\name#1{{\ulcorner{#1}\urcorner}}
\def\calP{{\mathcal P}}
\begin{document}

\sloppy

\title{Ideal Presentations and Numberings of Some Classes of Effective Quasi-Polish Spaces}

\author{Matthew de Brecht\\
        Graduate School of Human and Environmental Studies\\ Kyoto University, Japan 
\\       { \tt matthew@i.h.kyoto-u.ac.jp}
\\and
\\Takayuki Kihara\\
{\normalsize Department of Mathematical Informatics, Graduate School of Informatics}\\
{\normalsize Nagoya University, Japan}\\
{\normalsize {\tt kihara@i.nagoya-u.ac.jp}}%\\[5pt]
\\and
\\Victor Selivanov%\thanks{Thanks}
\\{\normalsize A.P. Ershov Institute of Informatics Systems SB RAS}\\
{\normalsize Novosibirsk, Russia}\\
{\normalsize {\tt vseliv@iis.nsk.su}}
}

%\titlerunning{On Ideal Representations of Effective Quasi-Polish Spaces}
%\author{Matthew de Brecht\thanks{De Brecht's research was supported by JSPS KAKENHI Grant Number 18K11166.
%}\inst{1}
%	\and\\
%	Takayuki Kihara\thanks{Kihara's research was partially supported by JSPS KAKENHI Grant Numbers 19K03602 and 21H03392, and the JSPS-RFBR Bilateral Joint Research Project JPJSBP120204809.}\inst{2}
%	\and\\
%	Victor  Selivanov\thanks{Selivanov's research was supported by the RFBR-JSPS Bilateral Joint Research Project 20-51-50001.}\inst{3}\orcidID{0000-0003-4316-0859}
%	}
%\institute{Graduate School of Human and Environmental Studies\\ Kyoto University, Japan 
%	\and
%    Department of Mathematical Informatics, Graduate School of Informatics\\
%   Nagoya University, Japan\\
%	\and
%     A.P. Ershov Institute of
%Informatics Systems SB RAS\\
%Novosibirsk, Russia\\
%	\email{matthew@i.h.kyoto-u.ac.jp}, \email{kihara@i.nagoya-u.ac.jp}, \email{vseliv@iis.nsk.su}
%}

\maketitle

\begin{abstract}

%[{\bf TODO: Combine two abstracts!}]

The well known ideal presentations of countably based domains were recently extended to  (effective) quasi-Polish spaces. Continuing these investigations, we explore some classes of effective quasi-Polish spaces. In particular, we prove an effective version of the domain-characterization of quasi-Polish spaces, describe  effective  extensions of quasi-Polish topologies, discover  natural numberings of classes of effective quasi-Polish spaces, estimate the complexity of the (effective) homeomorphism relation  and of some classes of spaces w.r.t. these numberings, and investigate degree spectra of continuous domains.

\end{abstract}

\section{Introduction}\label{intro}

%[{\bf TODO: Combine two introductions!}]

%\subsection{Introduction}

The investigation of computability in topological structures (which is currently a hot topic in computability theory) is  less straightforward than the investigation of computability in countable algebraic structures \cite{eg99,ak00}. A reason is that it is not clear how to capture the computability issues for a topological space (even if the space is Polish) by a single countable algebraic structure. Nevertheless, people often look for analogues of well-developed notions and methods of the computable structure theory in the topological context. For instance, analogues of computable categoricity turned out fruitful also in the study of computable metric spaces and Banach spaces (see e.g. \cite{Mel13,McSt19}), and analogues of degree spectra turned out interesting also for topological spaces \cite{s20,hks,HTMN}.

In this paper, which is a major update of the conference paper \cite{bks22}\footnote{Sections \ref{domchar}, \ref{ext}, \ref{dom}, \ref{enumerating-continuous-domains}, \ref{spectra} are entirely new, Section \ref{complete} is extended by  new facts, other sections remain essentially the same as in \cite{bks22}.}, we explore several classes of effective quasi-Polish  (EQP-) spaces. Quasi-Polish (QP-) spaces \cite{br} are a class of well-behaved countably based spaces that has many interesting characterisations and includes many spaces of interest in analysis and theoretical computer science, such as Polish spaces, $\omega$-continuous domains, and countably based spectral spaces.  The study  of effective versions of  QP-spaces was initiated in \cite{kk17,s15} and recently continued in \cite{br1,hoy,br2}. In particular, it was shown that some equivalent characterizations of QP-spaces become non-equivalent in the effective setting, and it is not obvious which of the resulted notions are the ``right'' ones.

Theorem 11 in \cite{br1} characterises the EQP-spaces (called there precomputable QP-spaces)  as the spaces of ideals of c.e.\ transitive relations on $\omega$ (see also Theorem 3 in \cite{br2} for a more direct proof). This characterisation is very much in the spirit of domain theory where similar characterisations of computable domains are important. It is a basic technical tool of our paper because it enables, in particular, to discover  natural numberings of classes of EQP-spaces and to estimate the complexity of the (effective) homeomorphism relation  and of some classes of spaces w.r.t. these numberings. 
Our investigation of numberings of classes of EQP-spaces is  analogous to the investigation of numberings of classes of algebraic structures popular in computable structure theory (see e.g. \cite{s03,gn,ffh12,ffn12,FHM14} and references therein).

Along with the mentioned results on numberings of classes of EQP-spaces, we obtain several other results. In particular, we prove an effective version of the domain-characterization of quasi-Polish spaces \cite[Theorem 53]{br}, describe  effective  extensions of quasi-Polish topologies (partially effectivising some results in \cite[Section 14]{br}),  and investigate degree spectra of continuous domains (continuing the investigation of degree spectra of algebraic domains in \cite{s20} and of Polish spaces in \cite{hks,HTMN}). 

After recalling some preliminaries in the next section,  we discuss in Section \ref{domchar} some variations of the ideal characterisation of EQP-spaces. In Section \ref{ext} we apply this to the problem of extending EQP topologies. In Section \ref{dom} we study ideal characterizations for classes of effective domains partly identified in \cite{s20}. This study  turns out closely related to the study of computable and c.e.\ partial orders and some other  binary relations  of interest to computability theory (see e.g. \cite{chol} and references therein). In particular, our discussion leads to an effective version of a domain-characterization of quasi-Polish spaces (\cite{br}, Theorem 53). In Section \ref{func} we describe representations of computable functions between some classes of effective domains.

In  Section \ref{num} we discuss natural numberings of some classes of c.e. binary relations on $\omega$ and of the corresponding classes of EQP-spaces (some of which were considered in \cite{s20}).
In Section \ref{enumerating-continuous-domains} we construct a computable numbering of c.e.~continuous domains (which gives, in a sense, a partial answer to the question of enumerability of c.e.~interpolable relations posed in the conference paper [6]); the existence of such a numbering  is much less straightforward than the corresponding results for algebraic domains.
In Section \ref{comisom}  we  estimate the complexity of  (effective) homeomorphism in the introduced numberings in parallel to the similar question for algebraic structures (see e.g. \cite{gn,ffh12,ffn12}).
In Section \ref{complete} we establish precise estimates of index sets of some popular classes of topological spaces related to separation axioms.
We conclude in Section \ref{spectra} by presenting several non-trivial facts about degree spectra of continuous domains which complement some results in \cite{s20,hks,HTMN}.

%In this paper we continue the  investigation of ideal representations of EQP spaces and of some classes of effective domains (including those recently identified in \cite{s20}). In particular, we show that in the ideal characterization of EQP spaces from \cite{br1} one can use computable transitive relations instead of c.e. ones, although similar variations for some versions of effective domains from \cite{s20} fail. We also prove an effective version of the domain-characterization of QP spaces in \cite{br},  describe  effective versions of extending quasi-Polish topologies, and of continuous functions between domains. Some of these results have interesting applications to characterizing algorithmic complexity of some properties of effective domains. 

\section{Preliminaries}\label{prel}

Here we recall some notation, notions and facts used throughout the paper. More special information is recalled in the corresponding sections below.

We use standard set-theoretical notation, in particular, $Y^X$ is the set of functions from $X$ to $Y$, and $P(X)$ is the class of subsets of a set $X$.
All (topological) spaces  in this paper are  countably based $T_0$ (cb$_0$-spaces, for short). We denote the homeomorphism relation by $\simeq$.
An {\em  effective space} is a pair $(X,\beta)$ where $X$ is a cb$_0$-space, and $\beta:\omega\to P(X)$ is a numbering of a base in $X$ such that there is a uniformly c.e.\ sequence $\{A_{ij}\}$ of c.e.\ sets with
$\beta(i)\cap\beta(j)=\bigcup\beta(A_{ij})$ where $\beta(A_{ij})$ is the image of $A_{ij}$ under $\beta$. We simplify $(X,\beta)$ to $X$ if $\beta$ is clear from the context. 
Any subspace $Y$ of an effective space $(X,\beta)$ is effective with the induced numbering $\beta_Y(n)=\beta(n)\cap Y$ of the base sets.
%The \emph{effectively open sets} in $X$ are the sets~$\bigcup\beta(W)$, for some c.e.~set~$W\subseteq\mathbb{N}$. The standard numbering $\{W_n\}$ of c.e.\ sets \cite{ro67} induces the numbering $n\mapsto\bigcup\beta(W_n)$ of the effectively open sets. The notion of effective space allows to define e.g. computable and effectively open functions between such spaces \cite{wei00,s15}. Any subspace $Y$ of an effective space $(X,\beta)$ is effective with the induced numbering $\beta_Y(n)=\beta(n)\cap Y$ of the base sets.

The effective space $(X,\beta)$ is {\em c.e.\ (or overt)} if the set $\{n\mid\beta(n)\neq\emptyset\}$ is c.e. A subspace of a c.e.\ space is not necessarily c.e.
Among the effective spaces are: the discrete space $\mathbb{N}$ of natural numbers, the Euclidean spaces ${\mathbb R}^n$, the Scott domain $P\omega$ (the powerset of the natural numbers with the Scott-topology; see \cite{aj} for information about domains), the Baire space $\mathcal{N}=\mathbb{N}^\mathbb{N}$, the Hilbert cube $[0,1]^\omega$; all these spaces come  with natural numberings of bases. 
With any effective  space $(X,\beta)$ we associate the {\em canonical embedding}~$e:X\to P\omega$ defined by $e(x)=\{n\mid x\in \beta(n)\}$. The canonical embedding is a computable homeomorphism between $X$ and the subspace $e(X)$ of $P\omega$. %Clearly, $A\in\Pi^0_2(X)$ iff $e(A)\in\Pi^0_2(P\omega)$.

In any effective space $X$, one can define effective versions of classical hierarchies (see e.g. \cite{s15}), in particular the effective Borel hierarchy $\{\Sigma^0_{1+n}(X)\}_{n<\omega}$ and the  effective Luzin hierarchy  $\{\Sigma^1_{n}(X)\}_{n<\omega}$. For $X=\omega$, these coincide resp. with the arithmetical and analytical hierarchies.

An effective space is {\em effective Polish} (resp. {\em effective quasi-Polish}, abbreviated as EQP) if it is effectively homeomorphic to a $\Pi^0_2$-subspace of the Hilbert cube (resp. the Scott domain).
Note that EQP-spaces are  called  in \cite{br1}  precomputable QP-spaces, while c.e.\ EQP-spaces are called computable QP-spaces.
All the aforementioned examples of spaces are c.e.~EQP-spaces.
Moreover, all computable Polish spaces and all computable $\omega$-continuous domains (see e.g. \cite{s15}) are c.e.\ EQP-spaces.
However, an effective Polish space is not necessarily computable Polish (nor even $\Delta^1_1$-computable Polish; for example, consider a $\Pi^0_1$ subspace $P$ of $\N^\N$ with no $\Delta^1_1$ elements. If $P$ were $\Delta^1_1$-overt, $P$ would have a $\Delta^1_1$ element).
The relation of effective homeomorphism between effective spaces will be denoted by $\simeq_e$. Note that if $(X,\beta)\simeq_e(Y,\gamma)$ then $X\simeq Y$. All classes of effective spaces considered below will be closed under $\simeq_e$. 

We use standard terminology about binary relations and about domains (see e.g. \cite{aj,s20}). In particular, an {\em ideal} of $(S;\rho)$ is a directed lower subset of $S$. By {\em interpolable} relations we mean transitive relations $\prec$ on $S$ such that any initial segment $\{x\mid x\prec y\}$, $y\in S$, is directed.

We briefly recall some terminology of domain theory (see e.g. \cite{aj,gh03,er93} for details). Let $X$ be a $T_0$-space. The {\em specialization order} $\leq_X$ on $X$ is defined by: $x\leq_Xy$, if $x\in U$ implies $y\in U$, for each open set $U\subseteq X$. Any continuous function $f:X\to Y$ is monotone w.r.t. $\leq_X,\leq_Y$. Let $K(X)$ be the set of {\em compact elements} of $X$, i.e. elements $x$ such that $\uparrow{x}=\{y\mid x\leq_Xy\}$ is open. The {\em approximation relation} $\ll_X$ on $X$ is defined by: $x\ll_Xy$, if $y$ is in the interior of $\uparrow{x}$. Clearly, $\ll_X\subseteq\leq_X$.
One can see that if $x\in X$ is a compact element, then $x\ll_Xy$ if and only if $x\leq_Xy$.
Note also that $x\in X$ is a compact element if and only if $x\ll_Xx$.

The space $X$ is a {\em continuous domain}, if $(X;\leq_X)$ is directed complete and there is a set $B\subseteq X$ (called a domain basis of $X$) such that the sets $\{x\mid b\ll_Xx\}$, $b\in B$, form a base of the topology in $X$, and every $x\in X$ is the directed supremum of $\{b\in B\mid b\ll_Xx\}$. An {\em $\omega$-continuous domain} is a continuous domain which possesses a countable domain base.
The space $X$ is an {\em algebraic domain}, if $(X;\leq_X)$ is directed complete and the sets $\uparrow{c}$, $c\in K(X)$, form a base of the topology in $X$. Every algebraic domain is also a continuous domain because $K(X)$ is a domain basis. An {\em $\omega$-algebraic domain} is an algebraic domain with  countably many compact elements.

The study of some classes of EQP-spaces  is closely related to the mentioned types of binary relations on $\omega$ with some effectivity conditions. Correspondingly, we assume the reader to be familiar with basics of computability theory, see e.g. \cite{ro67}. The notions of computable, c.e., and co-c.e.~binary relations are defined in a natural way. There are many interesting results about computable relations (as well as about computable and c.e.~structures in general). Below we will use the following nice fact (Theorem 2.1 in \cite{chol}) about effective partial orders: There is a co-c.e.~partial order on $\omega$ which is not isomorphic to any c.e.~partial order, and similarly with c.e.~and co-c.e.~interchanged.

We conclude this section with recalling the basic fact established in  Theorem 11 \cite{br1} (see also Theorem 3 in \cite{br2} for additional details).

\begin{definition}
Let $\prec$ be a transitive relation on $\omega$. A subset $I\subseteq \omega$ is an \emph{ideal} (with respect to $\prec$) if and only if:
\begin{enumerate}
\item
$I \not=\emptyset$,\hfill (\emph{$I$ is non-empty})
\item
$(\forall a \in I) (\forall b \in \omega)\, (b \prec a \Rightarrow b \in I)$,\hfill (\emph{$I$ is a lower set})
\item
$(\forall a,b \in I)(\exists c\in I)\, (a \prec c  \,\&\,  b \prec c)$.\hfill (\emph{$I$ is directed})
\end{enumerate}
The collection $\I{\prec}$ of all ideals has the topology generated by basic open sets of the form $[n]_{\prec} = \{ I \in \I{\prec} \mid n \in I\}$ for $n\in\omega$.
\qed
\end{definition}

As shown in \cite[Theorem 11]{br1}, such spaces of ideals are closely related to QP-spaces, namely: a space $X$ is quasi-Polish iff it is homeomorphic to $\I{\prec}$ for some transitive relation $\prec$ on $\omega$. Moreover,  an effective space $(X,\xi)$ is EQP iff it is computably homeomorphic to $\I{\prec}$ for some transitive c.e.\ relation $\prec$ on $\omega$.

\section{On ideal presentations of EQP-spaces}\label{domchar}

In this section we prove some  variations of Theorem 11 in \cite{br1}. 
A natural related question is: which class of EQP spaces is obtained if we restrict c.e. transitive relations above to, say, computable strict partial orders. It turns out that nothing new appears (although, as we show in Section \ref{dom}, similar variations for effective domains lead to new interesting notions).

%%%%%%%%%%%%%%%%%%%%%%%%%%%%%%%%%%%%%%%%%%%%%%%%%%
% BEGIN (Matthew 2021/01/11)
%%%%%%%%%%%%%%%%%%%%%%%%%%%%%%%%%%%%%%%%%%%%%%%%%%

\begin{proposition}\label{compstrict}
For every c.e.\ transitive relation $\prec$ on $\omega$, there is a computable strict partial order $\sqsubset$ on $\omega$ such that $\I{\prec}$ and $\I{\sqsubset}$ are computably homeomorphic.
\end{proposition}

\begin{proof} Let $\prec^{(k)}$ be a computable relation such that $x \prec^{(k)} y$ holds if and only if $x \prec y$ is accepted by some fixed Turing machine within $k$ steps. Let $P_{fin}\omega$ be the set of finite subsets of $\omega$, and for $F,G \in P_{fin}\omega$ and $m,n\in\omega$ define $\langle F,m\rangle \sqsubset \langle G,n\rangle$ if and only if the following all hold:
\begin{enumerate}
\item
$F \subseteq G$,
\item
$m < n$,
\item
$(\forall y\in F)(\forall x\leq n)\,[ x\prec^{(n)} y \Rightarrow x \in G]$,
\item
$(\exists y\in G)(\forall x\in F)\, x \prec^{(n)} y$.
\end{enumerate}
It is easy to verify that $\sqsubset$ is computable, irreflexive, and transitive, hence it is a computable strict partial order. Define $f\colon \I{\prec} \to \I{\sqsubset}$ as 
\[f(I) = \{ \langle F, m\rangle \mid F\subseteq I \text{ is finite} \,\&\, m\in\omega\},\]
and define $g\colon \I{\sqsubset} \to \I{\prec}$ as 
\[g(I) = \bigcup_{\langle F,m\rangle \in I} F.\]

It is straightforward to show that $f$ is well-defined and computable. We verify that $g$ is well-defined, and then computability will be obvious. Fix any $I \in \I{\sqsubset}$. It is clear that $g(I) \not=\emptyset$. To show $g(I)$ is a lower set, assume $x \prec y \in g(I)$. Choose $k$ large enough that $x\leq k$ and $x \prec^{(k)} y$, and fix $\langle F,m\rangle \in I$ with $y \in F$. Using the directedness of $I$ and $(2)$, there is $\langle G,n\rangle \in I$ with $\langle F,m\rangle \sqsubset \langle G,n\rangle$ and $k \leq n$. It follows from $(3)$ that $x\in G$, hence $x \in g(I)$. To prove that $g(I)$ is directed, note that for any $\langle F,m\rangle, \langle G,n\rangle\in I$, there is $\langle H,p\rangle \in I$ with  $\langle F,m\rangle \sqsubset \langle H,p\rangle$ and $\langle G,n\rangle \sqsubset \langle H,p\rangle$, so by $(4)$ there are $x,y\in H$ satisfying $(\forall w\in F)\, w \prec x$ and $(\forall w\in G)\, w \prec y$. Then again using directedness of $I$, there is $\langle H',p'\rangle\in I$ with $\langle H, p \rangle \sqsubset \langle H', p'\rangle$, hence $(4)$ and the transitivity of $\prec$ implies every pair of elements in $F\cup G$ has a $\prec$-upper bound in $H'$ which is contained in $g(I)$.

It only remains to show that $f$ and $g$ are inverses of each other. It is clear that $g(f(I)) = I$ for each $I\in \I{\prec}$. It is also clear that $I \subseteq f(g(I))$ for each $I\in \I{\sqsubset}$. To see that $f(g(I)) \subseteq I$, fix $\langle F,m\rangle \in f(g(I))$. For each $x\in F$ there is $\langle F_x,m_x\rangle\in I$ with $x \in F_x$. Let $\langle G, n\rangle$ be a $\sqsubset$-upper bound of $\{ \langle F_x, m_x \rangle \mid x \in F\}$ in $I$, and let $\langle H,p\rangle\in I$ be a $\sqsubset$-upper bound of $\langle G,n\rangle$. Then there exists $y\in H$ such that $x \prec^{(p)} y$ for each $x \in F$. Furthermore, from the directedness of $I$ and $(2)$, we can assume without loss of generality that $m < p$. It then follows that $\langle F,m\rangle \sqsubset \langle H,p \rangle\in I$, hence $\langle F,m\rangle\in I$ because $I$ is a lower set. This completes the proof that $\I{\prec}$ and $\I{\sqsubset}$ are computably homeomorphic.
\end{proof}

Now we relate the spaces of ideals $\I{\prec}$ where $\prec$ is transitive, to the spaces $\I{\sqsubseteq}$ where $\sqsubseteq$ is a partial order on $\omega$. Recall that an ideal of $\sqsubseteq$ is {\em principal} if it has a largest element w.r.t. $\sqsubseteq$.

\begin{proposition}\label{comppo}
For any c.e. transitive relation $\prec$ on $\omega$ there exists a computable partial order $\sqsubseteq$ on $\omega$ such that $\I{\prec}$ is computably homeomorphic to the subspace of non-principal ideals in $\I{\sqsubseteq}$.
\end{proposition}
\begin{proof}
From Proposition~\ref{compstrict} there is a computable strict partial order $\sqsubset$ on $\omega$ such that $\I{\prec}$ and $\I{\sqsubset}$ are computably homeomorphic. Define $\sqsubseteq$ to be the reflexive closure of $\sqsubset$. It suffices to show that any $I\subseteq \omega$ is an ideal of $\sqsubset$ if and only if $I$ is a non-principal ideal of $\sqsubseteq$.

Assume $I\in\I{\sqsubset}$. Clearly $I\not=\emptyset$. If $a \sqsubseteq b \in I$ then since $I$ is $\sqsubset$-directed there is $c\in I$ with  $b \sqsubset c$, hence $a \sqsubset c$ which implies $a \in I$ because $I$ is a $\sqsubset$-lower set. Therefore, $I$ is a $\sqsubseteq$-lower set.  Since $I$ is $\sqsubset$-directed it is clear that $I$ is $\sqsubseteq$-directed and also non-principal with respect to $\sqsubseteq$.

Conversely, assume $I\in\I{\sqsubseteq}$ is non-principal. It is immediate that $I$ is non-empty and a $\sqsubset$-lower set. Every $a,b \in I$ has a $\sqsubseteq$-upper bound $c\in I$, and since $I$ is non-principal there is $c'\in I$ with $c \subseteq c'$. Therefore, $c'$ is a $\sqsubset$-upper bound of $a$ and $b$, which proves $I$ is $\sqsubset$-directed.
\end{proof}

\section{Extending EQP topologies}\label{ext}

It is known that if countably many $\mathbf{\Delta}^0_2$-sets are added to the topology of a quasi-Polish space then the resulting space is again quasi-Polish \cite{br}. The following result is an effective version of this observation, but restricted to closed sets.

\begin{theorem}
Given a c.e.\ transitive relation $\prec$ on $\omega$ and a c.e.\ set $U \subseteq \omega$, one can effectively obtain a c.e.\ transitive relation $\sqsubset$ on $\omega$ such that $\I{\sqsubset}$ is computably homeomorphic to the space obtained by adding $A= \{ I \in \I{\prec} \mid (\forall x\in U)\, x\not\in I\}$ as a c.e.\ open set to the topology of $\I{\prec}$. 
\end{theorem} 
\begin{proof}
From Proposition~\ref{compstrict}, we can assume without loss of generality that $\prec$ is computable. We write $U^{(n)}$ for the subset of $U$ enumerated within $n$ steps by some fixed Turing machine. Let $\omega_* = \omega\cup\{*\}$, where $\{*\}$ is a symbol not in $\omega$. Define a c.e.\ transitive relation $\sqsubset$ on $(P_{fin}\omega_*)\times n$ as $\langle F,m\rangle \sqsubset \langle G,n\rangle$ if and only if the following all hold:
\begin{enumerate}
\item
$m<n$,
\item
$*\in F \Rightarrow *\in G$,
\item
$(\exists y\in G\setminus\{*\})(\forall x\in F\setminus\{*\})\, x \prec y$,
\item
$* \in G$ or $(\exists x\in U)\, x \in G$,
\item
$[*\in F \,\&\, y \in F\setminus \{*\}] \Rightarrow (\forall x\leq n)\,[x \in U^{(n)} \Rightarrow x \not\prec y]$.
\end{enumerate}
It is clear that $\sqsubset$ is c.e. The only non-trivial part of proving that $\sqsubset$ is transitive is verifying that $(5)$ holds, so we will show that here. Assume $\langle F,m\rangle \sqsubset \langle G,n\rangle \sqsubset \langle H,p\rangle$. Assume $*\in F$ and $y\in F\setminus\{*\}$ and that $x\leq p$ is in $U^{(p)}$. Then $(2)$ implies $*\in G$, and by $(3)$ there is $z\in G\setminus\{*\}$ with $y \prec z$, and by $(5)$ we have $x \not\prec z$. So if $x \prec y$ then transitivity would yield $x \prec z$, a contradiction. Therefore, $(5)$ holds, and we obtain $\langle F,m\rangle \sqsubset \langle H,p\rangle$.

Define $g\colon \I{\sqsubset} \to \I{\prec}$ as 
\[g(I) = \bigcup_{\langle F,m\rangle\in I} F\setminus\{*\}.\]
We first show that $g(I)$ is well-defined. $g(I)$ is non-empty because $(3)$ implies there is $\langle G,n\rangle \in I$ with some $y\in G\setminus\{*\}$. To show $g(I)$ is a lower set, assume $x \prec y \in g(I)$, and let $\langle F,m\rangle \in I$ be such that $y \in F\setminus\{*\}$. Since $I$ is directed, there is $\langle G,n\rangle \in I$ with $\langle F,m\rangle \sqsubset \langle G,n\rangle$. By $(3)$, there is $z\in G\setminus\{*\}$ with $y \prec z$, hence $x \prec z$ by transitivity of $\prec$. It follows that $\langle \{x\}, m \rangle \sqsubset \langle G, n\rangle$, which implies $\langle \{x\}, m \rangle\in I$ and therefore $x \in g(I)$. The proof that $g(I)$ is directed is identical to the proof that $g(I)$ is directed in Proposition~\ref{compstrict}.

It is clear that $g$ is computable. To see that $g$ is injective, assume $g(I) = g(J)$ and fix any $\langle F, m \rangle \in I$. This implies $F\setminus\{*\}\subseteq g(I) = g(J)$, hence there is a $\prec$-upper bound $y$ of $F\setminus\{*\}$ in $g(J)$. Then there is $\langle G, n\rangle \in J$ with $y \in G$, and by directedness we can assume $m < n$ and that $G$ satisfies $(4)$. We show $\langle F,m\rangle \sqsubset \langle G,n\rangle$. The criteria $(1)$, $(3)$, and $(4)$ are satisfied by the choice of $G$. To see that $(2)$ is satisfied, assume for a contradiction that $*\in F$ and $*\not\in G$. Since $G$ satisfies $(4)$, there must be $x \in U$ with $x\in G$, and by the directedness of $J$ there is $\langle H,p\rangle\in J$ with $\langle G,n\rangle\sqsubset \langle H,p\rangle$ and $x \leq p$ and $x\in U^{(p)}$. Then $(3)$ implies there is $z\in H$ with $x\prec z$, but then $(5)$ implies $\langle H,p\rangle$ can have no $\sqsubset$-upper bound in $J$, which contradicts the directedness of $J$. Therefore, $*\in F \Rightarrow *\in G$, hence $(2)$ is satisfied. Finally, if $*\in F$ and $F\setminus \{*\} \not=\emptyset$, then by directedness of $I$ we can find some $\langle H,p\rangle\in I$ with $\langle F,m\rangle \sqsubset \langle H,p\rangle$ and $n < p$, and it easily follows that $(5)$ holds. Therefore, $\langle F,m\rangle \sqsubset \langle G,n\rangle$, hence $\langle F,m\rangle \in J$ because $J$ is a lower set. The proof that $J\subseteq I$ is identical, hence $g$ is injective.

To see that $g$ is surjective, fix $J\in \I{\prec}$. We first consider the case that $J\in A$. Define 
\[I = \{ \langle F, m\rangle \mid F\subseteq J\cup\{*\} \text{ is finite} \,\&\, m\in\omega\}.\]
Then $I$ is clearly non-empty, and it is a lower set because if $\langle F,m\rangle \sqsubset \langle G,n\rangle \in I$, then $(3)$ implies $G$ contains a $\prec$-upper bound of $F$, hence $F\subseteq J$ because $J$ is a lower set. To see that $I$ is directed, assume $\langle F,m\rangle, \langle G,n\rangle \in I$, and let $y\in J$ be a $\prec$-upper bound of $(F\cup G)\setminus\{*\}$. Set $H = \{y, *\}$ and $p = m+n+1$. Then $\langle H,p\rangle \in I$ is a $\sqsubset$-upper bound of  $\langle F,m\rangle$ and $\langle G,n\rangle$. Thus $I\in \I{\sqsubset}$ and clearly $g(I)= J$. The case when $J\not\in A$ is similar, except we define
\[I = \{ \langle F, m\rangle \mid F\subseteq J \text{ is finite} \,\&\, m\in\omega\},\]
and when proving that $I$ is directed, we use the fact that there exists $x \in U\cap I$ when constructing a $\sqsubset$-upper bound of a pair of elements of $J$  so that it satisfies $(4)$. This completes the proof that $g$ is a computable bijection.

Finally, it is clear that every basic open subset of $\I{\sqsubset}$ is equal to $g^{-1}(O)$ or to $g^{-1}(A \cap O)$, where $O$ is an open subset of $\I{\prec}$.  Therefore, $\I{\sqsubset}$ is computably homeomorphic to the space obtained by adding $A$ as a c.e.\ open set to the topology of $\I{\prec}$. 
\end{proof}

Using the construction of countable products and equalizers described in \cite{br}, we can easily generalize the above theorem to handle c.e. sequences of co-c.e. closed sets.

\begin{corollary}\label{cor:extend_countable_closed}
Given a c.e.\ transitive relation $\prec$ on $\omega$ and a c.e.\ sequence  $(A_i)_{i\in\omega}$ of co-c.e. closed subsets of $\I{\prec}$, one can compute a c.e.\ transitive relation $\sqsubset$ on $\omega$ such that $\I{\sqsubset}$ is computably homeomorphic to the space obtained by adding each $A_i$ ($i\in\omega$) as a c.e.\ open set to the topology of $\I{\prec}$. 
\end{corollary}
\begin{proof}
Let $\sqsubset_i$ be the c.e. transitive relation obtained from the previous theorem for the space with $A_i$ joined to the topology of $\I{\prec}$, let $\sqsubset'$ be a c.e. encoding for the product of these spaces, and let $\sqsubset$ be a c.e. encoding for the diagonal of $\I{\sqsubset'}$, obtained as the equalizer of the computable functions $(x_i)_{i\in\omega}\mapsto (g_i(x_i))_{i\in\omega}$ and $(x_i)_{i\in\omega}\mapsto (g_0(x_0))^\omega$, where $g_i$ is the bijection from $\I{\sqsubset_i}$ to $\I{\prec}$ used in the proof of the above theorem. Then $\sqsubset$ satisfies the corollary.
\end{proof}

%%%%%%%%%%%%%%%%%%%%%%%%%%%%%%%%%%%%%%%%%%%%%%%%%%
% END  (Matthew 2021/01/11)
%%%%%%%%%%%%%%%%%%%%%%%%%%%%%%%%%%%%%%%%%%%%%%%%%%

\section{On ideal presentations of effective domains}\label{dom}

Here we establish ideal characterizations of some effective versions of domains. Let us first  recall definitions of two classes of effective domains (one of which is known while the other recently introduced in \cite{s20}). We warn the reader that the adjectives ``computable'' and ``c.e.'' are sometimes used in the literature inconsistently. 

\begin{definition}\label{effdomain}
By a {\em computable  domain}  we mean a
pair $(X,b)$ where $X$ is an $\omega$-continuous domain and $b:\omega\to X$ is a numbering of a domain base in $X$ such that the relation $b_i\ll_Xb_j$ is computable, where $\ll_X$ is the approximation relation on $X$.
The notion of c.e.\ domain is obtained by using ``c.e.'' instead of ``computable''.
\end{definition}

For  algebraic domains, we have at least three natural versions of effectiveness.

\begin{definition}\label{effalg}
\begin{enumerate}
\item By a {\em computable algebraic domain}  we mean a
pair $(X,c)$ where $X$ is an $\omega$-algebraic domain and $c:\omega\to X$ is a numbering of the compact elements in $X$ such that the relation $c_i\leq_Xc_j$ is computable, where $\leq_X$ is the specialization order on $X$.
\item By a {\em strongly c.e.\ algebraic domain}  we mean either a finite domain or a
pair $(X,c)$ where $X$ is an $\omega$-algebraic domain and $c:\omega\to X$ is a bijective numbering of the compact elements in $X$ such that the relation $c_i\leq_Xc_j$ is c.e.
\item By a {\em c.e.\ algebraic domain}  we mean a
pair $(X,c)$ where $X$ is an $\omega$-algebraic domain and $c:\omega\to X$ is a numbering of the compact elements in $X$ such that the relation $c_i\leq_Xc_j$ is c.e.
\end{enumerate}
\end{definition}

Notions (1) and (2) where introduced in \cite{s20} while notion (3) is sometimes met in the literature under the name ``computable $\omega$-algebraic domain''. It is easy to see that any computable algebraic domain is strongly c.e.\ and any strongly c.e.\ algebraic domain is c.e.
The next result is a reformulation of some well known facts of domain theory \cite{aj}. It was announced without a proof in \cite{s20} as Proposition 2; here we provide a proof, for the sake of completeness.

\begin{proposition}\label{idealdom}
\begin{enumerate}
\item A topological space is an $\omega$-continuous domain iff it is homeomorphic to $\I{\prec}$ for some transitive interpolable relation $\prec$ on $\omega$. 
\item A topological space is an $\omega$-algebraic domain iff it is homeomorphic to $\I{\sqsubseteq}$ for some preorder $\sqsubseteq$ on $\omega$.
\item An infinite topological space is an $\omega$-algebraic domain iff it is homeomorphic to $\I{\sqsubseteq}$ for some partial order $\sqsubseteq$ on $\omega$.
\end{enumerate}
\end{proposition}

\begin{proof}
(1) Let $\prec$ be a transitive interpolable relation  on $\omega$. It is well known (see e.g. Proposition 2.2.22 in \cite{aj}) that $\I{\prec}$ is an $\omega$-continuous domain. Conversely, let $X$ be an $\omega$-continuous domain and $b:\omega\to X$ be a numbering of a domain basis. Define the relation $\prec$ on $\omega$ by: $i\prec j$, if $b_i\ll_Xb_j$. Then $\prec$ is a transitive interpolable relation on $\omega$, and $X$ is homeomorphic to $\I{\prec}$ via $f(x)=\{i\mid b_i\ll_Xx\}$.

(2) Let $\sqsubseteq$ be a preorder on $\omega$. It is well known  that $\I{\sqsubseteq}$ is an $\omega$-algebraic domain the compact elements of which are the principal ideals   (note also that $\I{\sqsubseteq}$ is homeomorphic to $\I{\omega^*;\sqsubseteq^*}$, where $(\omega^*;\sqsubseteq^*)$ is the quotient-order of $(\om,\sqsubseteq)$). Conversely, let $X$ be an $\omega$-algebraic domain and $c:\omega\to X$ be a numbering of the compact elements. Define the preorder $\sqsubseteq$ on $\omega$ by: $i\sqsubseteq j$, if $c_i\leq_Xc_j$. Then  $X$ is homeomorphic to $\I{\sqsubseteq}$ via $f(x)=\{i\mid c_i\leq_Xx\}$.

(3) Similarly to (2).
 \end{proof}

We proceed with  effective versions of the above proposition for algebraic domains (for continuous domains such a direct effectivization is probably not known); see also Section \ref{categ}. The next fact was stated in \cite{s20} without proof.

\begin{proposition}\label{idealefalg}
\begin{enumerate}
\item An effective space is a computable (resp.  c.e.) algebraic domain iff it is effectively homeomorphic to $\I{\sqsubseteq}$ for some computable (resp.  c.e.) preorder $\sqsubseteq$ on $\omega$.
\item An infinite  effective space is a computable (resp. strongly c.e.) algebraic domain iff it is  effectively homeomorphic to $\I{\sqsubseteq}$ for some computable (resp.  c.e.) partial order $\sqsubseteq$ on $\omega$.
\end{enumerate}
\end{proposition}

\begin{proof}
(1) Let $\sqsubseteq$ be a computable preorder  on $\omega$. Then $(\I{\sqsubseteq};c)$, where $c_i=\downarrow i=\{j\mid j\sqsubseteq i\}$, is a computable domain because $c_i\subseteq c_j$ iff $i\sqsubseteq j$. Conversely, let $(X,c)$ be a computable algebraic domain. Define the computable preorder $\sqsubseteq$ on $\omega$ by: $i\sqsubseteq j$, if $c_i\leq_Xc_j$. Then $\sqsubseteq$ has the desired properties. 
The same argument works for the c.e.\ case.

(2) Let $\sqsubseteq$ be a computable partial order on $\omega$. Then $(\I{\sqsubseteq},c)$, where $c$ is as in item (1), is an infinite computable domain. Conversely, let $(X,c')$ be an infinite computable algebraic domain. Since $K(X)$ is infinite, it is straightforward to modify $c'$ and obtain an injective numbering $c:\omega\to X$ of the compact elements, preserving the computability of the relation $c_i\leq_Xc_j$. Then proceed as above.
A similar argument works for the strong c.e.\ case.
\end{proof}

%%%%%%%%%%%%%%%%%%%%%%%%%%%%%%%%%%%%%%%
% Added by Matthew (2023/01/10)
%%%%%%%%%%%%%%%%%%%%%%%%%%%%%%%%%%%%%%%
Although we assume the Scott topology on $\omega$-algebraic domains in this paper, the Lawson topology also has useful applications (see \cite{gh03}). The Lawson topology on a domain $X$ refines the Scott topology on $X$ by adding all sets of the form $X\setminus \uparrow x$ as open sets. For algebraic domains, it suffices to only add sets of the form $X\setminus \uparrow x$ for $x\in K(X)$ to obtain the Lawson topology (see the hint after Exercise~III-1.14 in \cite{gh03}). Using Corollary~\ref{cor:extend_countable_closed} and Proposition~\ref{idealefalg}, we obtain the following partial effectivization of the known fact that an $\omega$-algebraic domain with the Lawson topology is a (zero-dimensional) Polish space.

\begin{corollary}
If $X$ is a computable (or c.e.) algebraic domain, then $X$ with the Lawson topology is an effective quasi-Polish space.
\qed
\end{corollary}

%%%%%%%%%%%%%%%%%%%%%%%%%%%%%%%%%%%%%%%
%%%%%%%%%%%%%%%%%%%%%%%%%%%%%%%%%%%%%%%

Next we prove an effective version of the following domain-characterization of quasi-Polish spaces established in \cite{br}: a space is quasi-Polish iff it is homeomorphic to the space of non-compact elements of an $\omega$-algebraic (equivalently, of an $\omega$-continuous) domain.

\begin{theorem}\label{idealefdom}
For an effective space $(X,\xi)$ the following are equivalent:
\begin{enumerate}
\item $(X,\xi)$ is an effective quasi-Polish space. 
\item $(X,\xi)$ is computably homeomorphic to the space of non-compact elements of a computable algebraic domain.
\item $(X,\xi)$ is computably homeomorphic to the space of non-compact elements of a computable domain.
\end{enumerate}
\end{theorem}

\begin{proof}
(1)$\to$(2). By  Proposition \ref{compstrict}, $X$ is computably homeomorphic to $\I{\prec}$ for some computable strict partial order $\prec$ on $\omega$. By Proposition \ref{comppo}, $X$ is computably homeomorphic to the subspace of non-principal ideals in $\I{\sqsubseteq}$, for some computable partial order $\sqsubseteq$ on $\omega$. By Proposition \ref{idealefalg}(1), $\I{\sqsubseteq}$ is a computable algebraic domain. 

The implication (2)$\to$(3) is obvious since every computable algebraic domain is a computable domain.

(3)$\to$(1). Let $(Y,b)$ be a computable domain and let $X$ be  computably homeomorphic to $Y\setminus K(Y)$. Since $Y$ is EQP, it suffices to show that $K(Y)\in\Sigma^0_2(Y)$ \cite{br1,hoy}. The set $C=\{i\mid b_i\ll_Yb_i\}$ is computable and $K(Y)=\{b_i\mid i\in C\}$, hence it suffices to check that $\{b_i\}\in\Sigma^0_2(Y)$ uniformly on $i\in C$. We have
 $$\{b_i\}=\{y\mid b_i\leq_Yy\}\setminus\{y\mid y\not\leq_Yb_i\}=\{y\mid b_i\leq_Yy\}\setminus\bigcup_{b_j\not\prec_Y b_i} \{y\mid b_j\prec_Y y\}.$$
 This is because $y\leq_Yb_i$ iff, for any $j$, $b_j\ll_Yy$ implies $b_j\ll_Yb_i$.
 Since the relation $b_j\not\ll_Y b_i$ is computable, the desired estimate follows.
 \end{proof}

%Below we show that the analogues of Theorem \ref{idealefdom} for the other effective versions of domain fail.
%\section{Some separations}\label{separ}

Next we show that  the  effective versions of algebraic domains introduced above are non-equivalent. For this we use Theorem 2.1 in \cite{chol} cited in Section \ref{prel}, and Theorem 1 in \cite{s20} equivalent to following fact about c.e.\ preorders on $\omega$.

\begin{proposition}\label{cepre}
There is a c.e.\ preorder $\sqsubseteq$ on $\omega$ whose quotient-order $(\omega^*;\sqsubseteq^*)$ is infinite and not isomorphic to any c.e.\ partial order on $\omega$. 
\end{proposition}

We use this proposition to obtain the following fact which was also announced in \cite{s20} without proof:
 
\begin{proposition}\label{separalg}
\begin{enumerate}
\item There is a strongly c.e.\ algebraic domain which is not homeomorphic to any computable algebraic domain. 
\item There is a c.e.\ algebraic domain which is not homeomorphic to any strongly c.e.\ algebraic domain.
\item  There is a c.e.\ domain which is not homeomorphic to any computable  domain.
\end{enumerate}
 \end{proposition}

\begin{proof} (1) By Theorem 2.1 in \cite{chol}, there is a c.e.\ partial order $\sqsubseteq$ on $\omega$ which is not isomorphic to any computable partial order on $\omega$. We claim that the space $X=\I{\sqsubseteq}$ has the desired properties. Indeed, by Proposition \ref{idealefalg}(2) $X$ (with a natural numbering of the compact elements) is a strongly c.e.\ algebraic domain. Suppose, for a contradiction, that $X$ is homeomorphic to a computable algebraic domain $(Y,c)$, and let $f:X\to Y$ be a homeomorphism. As noticed in  the proof of Theorem 2 in \cite{s20}, the restriction of $f$ to $K(X)$ is an isomorphism between $(K(X);\leq_X)$ and $(K(Y);\leq_Y)$. By Proposition \ref{idealefalg}, $(K(X);\leq_X)$ and $(K(Y);\leq_Y)$ are isomorphic respectively to $(\omega,\sqsubseteq)$ and to $(\omega,\sqsubseteq_Y)$, for some computable partial order  $\sqsubseteq_Y$ on $\omega$. Thus, the latter two partial orders are isomorphic. A contradiction.

(2) Let  $\sqsubseteq$ be the preorder on $\omega$ from Proposition \ref{cepre}.  The same argument as in item (1) shows that the space $X=\I{\sqsubseteq}$ has the desired properties because  $(K(X);\leq_X)$ is isomorphic to
$(\omega^*;\sqsubseteq^*)$ by Proposition \ref{idealefalg}.

(3)  Let $(X,c)$ be the space from the proof of item (1), then $(X,c)$ is an infinite c.e.\ domain. Suppose for a contradiction that $X$ is homeomorphic to $Y$ for some computable domain $(Y,b)$. Let $C=\{i\mid b_i\ll_Yb_i\}$, then, as a base must contain all compact elements, $C=\{i_0<i_1<\cdots\}$ is an infinite computable set and $(Y,\{b_{i_j}\}_j)$ is a computable algebraic domain. This contradicts  item (1). 
 \end{proof}

\section{Representing computable functions}\label{func}

Here we describe useful representations of computable functions between EQP spaces and effective domains.

First we briefly recall some notions and facts from Section 2.2.6 of \cite{aj}. Let $\mathbb{A}=(A;\prec_1)$ and $\mathbb{B}=(B;\prec_2)$ be interpolable relations on arbitrary sets $A,B$; we denote the elements of $A$ as $a,a',\ldots$, and similarly for $B$. We say that  a set $R\subseteq A\times B$  is a morphism from $\mathbb{A}$ to $\mathbb{B}$ (in symbols, $R:\mathbb{A}\to\mathbb{B}$) if it satisfies the following conditions: if $aR b$ and $a\prec_1a'$ then $a'R b$; if $aR b$ and $b'\prec_2b$ then $aR b'$; for any $a$ there is $b$ with  $aR b$; for all $a,b,b'$ with $aR b,aR b'$ there is $b''$ with $b\prec_2b'',b'\prec_2b''$, and $aR b''$; if $aR b$ then $a'R b$ for some $a'\prec_1a$  (cf. Definition 2.2.27 in \cite{aj}).  

By Theorem 2.2.28 in \cite{aj}, the category $\inter$ of the interpolable relations $\mathbb{A}$ and the just defined morphisms is equivalent to the category $\cont$ of continuous domains as objects and continuous functions as morphisms. The equivalence is given by the functors $I:\inter\to\cont$ and $B:\cont\to\inter$ defined as follows. Let $I(\mathbb{A})$ be the space of ideals of $\mathbb{A}$, and, for $R:\mathbb{A}\to\mathbb{B}$, $I(R):I(\mathbb{A})\to I(\mathbb{B})$ be the image map $\lceilx R\rceilx(J)=\{b\mid\exists a\in J(aRb)\}$, $J\in I(\mathbb{A})$. For an object $X$ of $\cont$ let $B(X)=(X;\ll_X)$, and for a morphism $f:X\to Y$ of $\cont$ let $B(f)=\{(x,y)\mid y\ll_Yf(x)\}$. 

By Theorem 2.2.29 in \cite{aj}, the full subcategory $\pre$ of $\inter$ formed by the arbitrary preorders $\mathbb{A}=(A;\sqsubseteq)$ as objects,  is equivalent to the full subcategory $\alg$ of $\cont$ formed by the algebraic domains as objects. The equivalence is given by the functors $I_\pre:\pre\to\alg$ and $K:\alg\to\pre$ where $I_\pre$ is the restriction of $I$ to $\pre$, $K(X)$ is the restriction of $\leq_X$ to the set of compact elements in $X$,  and, for a morphism $f:X\to Y$ of $\alg$, let $K(f)=\{(x,y)\mid x\in K(X)\wedge y\in K(Y)\wedge y\leq_Yf(x)\}$.

Now we describe effective versions of the cited results for the classes of effective algebraic domains from Proposition \ref{idealefalg}. The effectivization of Theorem 2.2.28 in \cite{aj} is currently not clear. The problem is that the functor $I$ increases the algorithmic complexity (while the effective version of the functor $B$, $B_e(X,b)=(\omega;\prec)$ where $m\prec n\leftrightarrow b(m)\ll_Xb(n)$, preserves the algorithmic complexity).

Let $\alg_e$ (resp. $\alg_s$, resp. $\alg_c$) be the category of c.e. (resp. infinite strongly c.e., resp. computable) algebraic  domains as objects and the computable functions as morphisms. Let  $\pre_{e}$ (resp. $\pre_{s}$, resp. $\pre_{c}$) be the category of c.e. preorders (resp.  c.e. partial orders, resp. computable preorders) on $\om$; the morphisms in all three categories are c.e. morphisms of $\pre$ restricted to the objects of these categories. Let $\simeq$ denote the equivalence of categories.

\begin{theorem}\label{categ}
We have: $\pre_{e}\simeq\alg_e$, $\pre_{s}\simeq\alg_s$, $\pre_{c}\simeq\alg_c$.
 \end{theorem}
 
\begin{proof} We only define the functors witnessing the equivalences, leaving the straightforward checking of their properties to the reader. Define $I_{e}:\pre_{e}\to\alg_e$ as follows: $I_{e}(\omega;\sqsubseteq)=(I(\sqsubseteq),b)$ where $b(n)=\{m\mid m\sqsubseteq n\}$, and for a morphism $R:(\om;\sqsubseteq_1)\to(\om;\sqsubseteq_2)$ let $I_{e}(R)$ be the image map $\lceilx R\rceilx$ restricted to $I(\sqsubseteq_1)$. Define $B_{e}:\alg_e\to\pre_{e}$ as follows: if $(X,b)$ is a c.e. algebraic domain then let $B_{e}(X,b)=(\omega;\sqsubseteq)$ where $m\sqsubseteq n\leftrightarrow b(m)\leq_Xb(n)$, and if $f:(X,b)\to(Y,c)$ is a computable function then let $B_{e}(f)=\{(m,n)\mid c(n)\leq_Yf(b(m))\}$. Then $I_{e},B_{e}$ are witnesses for $\inter_{e}\simeq\alg_e$ The witnesses   for the remaining two equivalences are obvious restrictions of  $I_{e},B_{e}$.
 \end{proof}

We conclude this section with remarks on representing functions between QP-spaces represented as spaces of ideals $\I{\prec_1}$ and $\I{\prec_2}$. We define a code for a partial function to be any subset $R\subseteq \omega\times\omega$. Each code $R$ encodes the partial function $\name{R}:\subseteq \I{\prec_1}\to \I{\prec_2}$ defined as
\begin{eqnarray*}
\name{R}(I) &=& \{ n\in\omega \mid (\exists m\in I)\,\langle m,n\rangle\in R\},\\
dom(\name{R}) &=& \{ I \in \I{\prec_1} \mid \name{R}(I) \in \I{\prec_2}\}.
\end{eqnarray*}

The following fact  is Theorem 2 from  \cite{br2}\footnote{In the original paper, the statement of the theorem incorrectly omitted the requirement that $\prec_1$ be a c.e. relation. We are grateful to Ivan Georgiev for pointing out this mistake and providing a counter example.}.

\begin{theorem}\label{theorem:cont_func_code}
Let $\prec_1$ and $\prec_2$ be c.e. transitive relations on $\omega$. A total function $f\colon \I{\prec_1} \to \I{\prec_2}$ is computable if and only if there is a c.e. code $R\subseteq \omega\times\omega$ such that $f=\name{R}$.
\qed
\end{theorem}

%For convenience, we will often consider relations on countable sets, with the assumption that they are suitably encoded as (subsets) of $\omega$.

%First we recall the representation of computable functions $f:I(\prec_1)\to I(\prec_2)$ between spaces of ideals of c.e. transitive relations on $\om$ established in \cite{br1}, Theorem 2. We associate with any $R\subseteq\mathbb{N}\times\mathbb{N}$ a partial function $\lceilx R\rceilx$ from $I(\prec_1)$ to $I(\prec_2)$ as follows: $\lceilx R\rceilx(I)=\{n\in\mathbb{N}\mid\exists m\in I(mRn)\}$, $dom(\lceilx R\rceilx)=\{I\in I(\prec_1)\mid\lceilx R\rceilx(I)\in I(\prec_2)\}$. By Theorem 2 in \cite{br1},  $f:I(\prec_1)\to I(\prec_2)$ is a computable partial functions iff $f=\lceilx R\rceilx$ for some c.e.  binary relation $R$. Unfortunately, in this case we do not know a version of Theorem \ref{categ}.

%The notions from the previous paragraph may be in the obvious way extended to the spaces of ideals $\mathbb{A}$, where $\mathbb{A}=(A;\prec)$ is a transitive relation on arbitrary (possibly, uncountable) set $A$. Such spaces suggest a natural extension of QP spaces to non-countably-based spaces. We leave as an open question to investigate how natural is this extension, in particular which of the known characterizations of QP spaces survive under this extension.   

\section{Enumerating classes of spaces}\label{num}

Here we introduce and study numberings of some classes of relations on $\omega$ and of EQP-spaces.
Some natural numberings of spaces may be defined directly from the definitions of Section \ref{prel}. For any effective space $X$, let $\pi_X$ be the standard numbering of $\Pi^0_2$-subspaces of $X$. In the particular case $X=P\omega$ we obtain the numbering $\pi=\pi_X$ of all (up to $\simeq_e$) EQP-spaces. In the particular case $X=[0,1]^\omega$, $\pi_X$ is a numbering of all effective Polish spaces (because, up to homeomorphism, Polish spaces are precisely the $\mathbf{\Pi}^0_2$-subspaces of the Hilbert cube, see e.g. Theorem 4.14 in \cite{ke95}); setting $\mu(n)=e(\pi_X(n))$, where $e$ is the canonical embedding of $[0,1]^\omega$ into $P\omega$, we obtain a numbering $\mu$ of effective Polish spaces realised as $\Pi^0_2$-subspaces of $P\omega$. 

Other natural numberings of spaces are defined  using the ideal representations. 
We first define some numberings of classes of relations on $\omega$. Setting $V_n=\{(i,j)\mid\langle i,j\rangle\in W_n\}$, we obtain a standard computable numbering $\{V_n\}$ of the class $\mathbf{E}$ of all c.e.~binary relations on $\omega$. Let $\mathbf{T},\mathbf{I},\mathbf{P},\mathbf{O}$ be the classes of all transitive c.e.\ relations, all interpolable c.e.\ relations, all c.e.\ preorders, and all c.e.\ partial orders on $\omega$, respectively.

\begin{proposition}\label{numb_rel}
\begin{enumerate}%\itemsep-1mm
\item There is a computable function $t$ such that: $V_{t(n)}\in\mathbf{T}$,  $V_n\in\mathbf{T}$ implies  $V_n=V_{t(n)}$, and $V_m=V_n$ implies $V_{t(m)}=V_{t(n)}$. 
\item There is a computable function $p$ such that: $V_{p(n)}\in\mathbf{P}$,  $V_n\in\mathbf{P}$ implies  $V_n=V_{p(n)}$, and $V_m=V_n$ implies $V_{p(m)}=V_{p(n)}$. 
\item There is a computable function $o$ such that: $V_{o(n)}\in\mathbf{O}$, and $V_n\in\mathbf{O}$ implies  $V_n=V_{o(n)}$. 
%\item There is a $\emptyset'$-computable function $i$ such that: $V_{i(n)}\in\mathbf{I}$, and $V_n\in\mathbf{I}$ implies  $V_n=V_{i(n)}$. 
 \end{enumerate}
\end{proposition}

\begin{proof}
(1) As $t$ we can take arbitrary computable function such that $V_{t(n)}$ is the transitive closure of $V_n$ (such a function obviously exists). 

(2) As $p$ we can take arbitrary computable function such that $V_{p(n)}$ is the reflexive transitive closure of $V_n$ (such a function obviously exists).

(3) Given a computable step-wise enumeration of $\{V_{p(n)}\}$, it is straightforward  to construct a computable sequence $\{A_n\}$ of c.e.\ partial orders on $\omega$ such that: $A_n\subseteq V_{p(n)}$; if $V_{p(n)}$ is a partial order then $A_n=V_{p(n)}$; if $V_{p(n)}$ is not a partial order then almost all elements of $A_n$ are pairwise incomparable. As $o$ we can take arbitrary computable function such that $V_{o(n)}=A_n$.
 \end{proof}

We thank an anonymous referee of the conference version of this paper for showing that there is no function $o$ as in item (3) with the additional property that $V_m=V_n$ implies $V_{o(m)}=V_{o(n)}$.

%We do not know whether there exists a function $o$ as in item (3), with the additional property that $V_m=V_n$ implies $V_{o(m)}=V_{o(n)}$.%\footnote{One of the referee's reports sketches a proof that there is no function $o$ with this additional property.}

\begin{corollary}\label{computable}
The classes $\mathbf{T},\mathbf{P},\mathbf{O}$ have computable numberings, namely the numberings $\{V_{t(n)}\}$, $\{V_{p(n)}\}$, $\{V_{o(n)}\}$, respectively.
\end{corollary} 

Compared to other classes, the method of enumerating the class $\mathbf{I}$ of interporable c.e.~relations is not at all clear.
The complexity of interpolability is one of the reasons why it is difficult to enumerate $\mathbf{I}$ in a simple way.

\begin{proposition}\label{prop:interpolable-pi-0-2-complete}
Deciding whether a given c.e.~transitive relation is interpolable is $\Pi^0_2$-complete.
\end{proposition}

\begin{proof}
It is clear that this decision is $\Pi^0_2$.
For completeness, consider the standard strict ordering $<_\mathbb{Q}$ of the rational numbers $\mathbb{Q}=\{q_i:i\in\N\}$, which is clearly interpolable.
For any $\Pi^0_2$ formula $\varphi(n)\equiv\forall a\exists b\theta(n,a,b)$, consider the restriction of $<_\mathbb{Q}$ to $Q_n=\{q_i:\forall a<i\exists b\theta(n,a,b)\}$, which is c.e.
If $\varphi(n)$ is true, then $Q_n=\mathbb{Q}$; otherwise $Q_n$ is finite.
Note that the strict order $<_\mathbb{Q}$ restricted to a finite set cannot be interpolable.
\end{proof}

We do not know whether the class $\mathbf{I}$ has a computable numbering but by Proposition \ref{prop:interpolable-pi-0-2-complete} we can define a natural non-computable one $\{V_{i(n)}\}$ where $i$ is the $\emptyset''$-computable function which enumerates the $\Pi^0_2$-set $\{m\mid V_{t(m)}\in\mathbf{I}\}$ in the increasing order.
For the moment, let $j,c$ be $\emptyset''$-computable functions which enumerate the $\Sigma^0_3$-sets $\{m\mid V_{t(m)}\text{ is computable}\}$ and  $\{m\mid V_{p(m)}\text{ is computable}\}$, respectively.
We return to this issue again in Section \ref{enumerating-continuous-domains} to show that a subclass of $\mathbf{I}$ large enough to represent all homeomorphism types of effective $\omega$-continuous domains has a computable numbering.

Theorem 11 in \cite{br1}, Corollary \ref{computable}, and Propositions 3,4 in \cite{s20} imply that $\{\I{V_{t(n)}}\}$,  $\{\I{V_{p(n)}}\}$, $\{\I{V_{o(n)}}\}$, $\{\I{V_{c(n)}}\}$ are  numberings of all (up to $\simeq_e$) EQP-spaces,  positive \ algebraic domains,   c.e.\ algebraic domains, and  computable\ algebraic domains,  respectively (see \cite{s20} for precise definitions and a discussion of these  classes of domains); we sometimes denote these numberings by $\iota,\alpha,\beta,\gamma$, respectively.
Sequences $\{\I{V_{i(n)}}\}$ and $\{\I{V_{j(n)}}\}$ are  numberings of natural classes of $\omega$-continuous domains, which we also denote  by $\delta$ and $\varepsilon$, respectively. Below is a summary of the introduced numberings.

%Recall  that $\pi$ and $\mu$ are the standard numberings of ${\Pi}^0_2$-subspaces of $\mathcal{P}\om$ and $[0,1]^\om$, respectively.

\begin{itemize}
\item
$\mu$: Standard numbering of ${\Pi}^0_2$-subspaces of $[0,1]^\om$.
\item
$\pi$: Standard numbering of ${\Pi}^0_2$-subspaces of $\mathcal{P}\om$.
\item
$\iota$: Numbering of EQP-spaces derived from the computable numbering  $\{V_{t(n)}\}$ of c.e. transitive relations ($\mathbf{T}$).
\item
$\alpha$: Numbering of positive algebraic domains derived from the computable numbering $\{V_{p(n)}\}$ of  c.e. preorders ($\mathbf{P}$).
\item
$\beta$: Numbering of c.e. algebraic domains derived from the computable numbering $\{V_{o(n)}\}$ of c.e. partial orders ($\mathbf{O}$).
\item
$\gamma$: Numbering of computable algebraic domains derived from the $\emptyset''$-computable numbering of computable partial orders.
\item
$\delta$: Numbering of $\omega$-continuous domains derived from the $\emptyset''$-computable numbering $\{I(V_{i(n)})\}$ of interpolable c.e.\ relations ($\mathbf{I}$).
\item
$\varepsilon$: Numbering of $\omega$-continuous domains derived from the $\emptyset''$-computable numbering $\{I(V_{j(n)})\}$ of interpolable computable relations.
\end{itemize}

The next proposition compares the introduced numberings under the following preorder on the numberings of effective spaces: $\nu\leq_e\nu'$, if $\nu(n)\simeq_e\nu'(f(n))$ for some computable function $f$; let $\equiv_e$ be the equivalence relation induced by $\leq_e$. For an oracle $h$, let $\leq^h_e$ and $\equiv^h_e$ be the $h$-relativizations of $\leq_e$ and $\equiv_e$, respectively. The presence of oracles in some of the reductions below is explained by the fact that numberings $\gamma,\delta,\varepsilon$ are defined in a less constructive way than the other numberings. 

\begin{proposition}\label{numb_compare}
We have: $\mu\leq_e\pi\equiv_e\iota$, $\beta\leq_e\alpha\leq_e\iota$, $\varepsilon\leq^{\emptyset''}_e\delta\leq^{\emptyset''}_e\iota$, and $\gamma\leq^{\emptyset''}_e\alpha$. The binary operations of product and coproduct are represented by computable functions in any of the  numberings $\mu,\pi,\iota,\alpha,\gamma$ (again, up to $\equiv_e$).
 \end{proposition}
 
\begin{proof} The relation $\pi\equiv_e\iota$ follows from the effectivity of proofs of Theorem 11 in \cite{br1} and Theorem 3 in \cite{br2}. The relation $\mu\leq_e\pi$ follows from Theorem 1 in \cite{hoy} because the Hilbert cube is a computable Polish space. The remaining relations follow from  the definition of the numbering and of functions $i,j,c$, and from Proposition \ref{numb_rel}. The assertion about product and coproduct is checked in a straightforward way, similar to Sections 3.1 and 3.2 in \cite{br2}.
 \end{proof}

May the non-computable numberings $\gamma,\delta,\epsilon$ be improved to computable numberings of the corresponding classes of EQP-spaces? In Section \ref{enumerating-continuous-domains} we give a positive answer for the case of $\delta$.

%%%%%%%%%%%%%%%%%%%%%%%%%%%%%%%%%%%%%%%%%%%%%%%%%

\section{Enumerating continuous domains}\label{enumerating-continuous-domains}

As shown in Proposition \ref{idealdom}, $\omega$-continuous domains can be represented by interpolable transitive relations.
However, as described in Section \ref{num} (also in \cite{bks22}), it is not at all clear how to provide a computable numbering for all c.e.~interpolable relations.
In this section, we address this issue.
Although it is still unknown whether there exists a computable enumeration of all c.e.~interpolable transitive relations on $\N$, we will show in this section that there exists a computable enumeration of the corresponding $\omega$-continuous domains (up to computable homeomorphism).
This was unknown at the time of \cite{bks22}, and we present the result here for the first time.

\begin{lemma}\label{lem:relation-map-to-homeo}
Let $(X, \prec_X)$ and $(Y, \prec_Y)$ be sets equipped with (c.e.) transitive relations. Assume $f\colon X\to Y$ is a (computable) function satisfying
\begin{itemize}
\item
$f$ is surjective, and
\item
$x \prec_X x' \iff f(x) \prec_Y f(x')$.
\end{itemize}
Let $G_f = \{ \langle x, f(x) \rangle \mid x \in X\}$ be the graph of $f$. Then $\name{G_f}\colon \I{\prec_X}\to\I{\prec_Y}$ is a (computable) homeomorphism.
\end{lemma}
\begin{proof}
First we show that $x \in I$ if and only if $f(x) \in \name{G_f}(I)$. If $x \in I$ then $f(x) \in \name{G_f}(I) $ by definition of $\name{G_f}$. Conversely, if $f(x) \in \name{G_f}(I)$, then there must be some $x' \in I$ with $f(x')=f(x)$. By directedness of $I$ there is $x'' \in I$ with $x' \prec_X x''$. Then $f(x) = f(x') \prec_Y f(x'')$ hence our assumption on $f$ implies $x \prec_X x''$, and it follows from $I$ being a lower set that $x \in I$.

Next we show that $\name{G_f}(I) \in \I{\prec_Y}$ for each $I\in\I{\prec_X}$. We use the assumption that $f$ is surjective in the following.
\begin{enumerate}
\item
(\emph{$\name{G_f}(I)$ is non-empty}): There is $x\in I$, hence there is $f(x) \in \name{G_f}(I)$.
\item
(\emph{$\name{G_f}(I)$ is a lower set}): Assume $f(x) \prec_Y f(x') \in \name{G_f}(I)$. Then $x \prec_X x' \in I$, hence $x\in I$. Therefore, $f(x) \in \name{G_f}(I)$.
\item
(\emph{$\name{G_f}(I)$ is directed}): Assume $f(x_0), f(x_1) \in \name{G_f}(I)$. Then $x_0,x_1 \in I$ so there is $x\in I$ with $x_0 \prec_X x$ and $x_1 \prec_X x$. Therefore, $f(x) \in\name{G_f}(I)$ and $f(x_0) \prec_Y f(x)$ and $f(x_1) \prec_Y f(x)$.
\end{enumerate}
It follows from Theorem~\ref{theorem:cont_func_code} that  $\name{G_f}\colon \I{\prec_X}\to\I{\prec_Y}$ is a total computable function.

Next we show that $\name{G_f}$ is surjective. Given $J \in \I{\prec_Y}$, set $I = f^{-1}(J)$. Then
\begin{enumerate}
\item
$I\not=\emptyset$ because $J \not=\emptyset$ and $f$ is surjective.
\item
$I$ is a lower set, because if $x \prec_X x' \in I$ then $f(x) \prec_Y f(x') \in J$, hence $f(x) \in J$, which implies $x\in I$.
\item
$I$ is directed, because if $x_0, x_1 \in I$, then there is $y \in J$ with $f(x_0) \prec_Y y$ and $f(x_1) \prec_Y y$, so from the surjectivity of $f$ there is $x \in I$ with $f(x) = y$ and satisfying $x_0 \prec_X x$ and $x_1 \prec_X x$.
\end{enumerate}

Finally, since $I \in [x]_{\prec_X}$ if and only if $\name{G_f}(I) \in [f(x)]_{\prec_Y}$, it follows that $\name{G_f}$ is a computable homeomorphism.
\end{proof}

\begin{theorem}
There exists a computable enumeration of all EQP $\omega$-continuous domains.
\end{theorem}
\begin{proof}
Let $(Y, \prec_Y)$ be a set equipped with a c.e. transitive relation. We will define a c.e. subset $X\subseteq \N$ and c.e. interpolable transitive relation $\prec_X$ on $X$ and a computable function $f\colon X\to Y$. We think of $2n \in \N$ as corresponding to $n \in Y$, and view $2n+1 \in \N$ as being a \emph{dummy symbol} $w_n$.

Let $\prec_Y^{(s)}$ be a finite approximation of $\prec_Y$ at stage $s$. Let $(F_n,y_n)_{n\in\N}$ be an enumeration of $\calP_{fin}(\N)\times \N$, such that $x \in F_n$ implies $x <n$. Define the partial function $f^{(0)} :\subseteq \N \to Y$ as having only even numbers in its domain, and set $f^{(0)}(2n) = n$. Let $\prec_X^{(0)}$ be the empty binary relation on $\N$.

At each stage $s$, each dummy symbol $w_k$ is \emph{inactive} at stage $0$, but may become \emph{active} at a later stage. An active dummy symbol $w_k$ may later become \emph{replaced}. A pair $(F_n,y_n)$ \emph{requires attention} at stage $s$ if each element of $F_n \cup\{y_n\}$ is active (and/or replaced) at stage $s$ and furthermore $x \prec_X^{(s)} y_n$ holds for each $x \in F_n$. Then we say that $z$ solves such $(F_n,y_n)$ if $x \prec_X^{(s)} z \prec_X^{(s)} y_n$ holds for each $x \in F_n$.

Do each of the following substages for each stage $s\geq 0$.
\begin{itemize}
\item
Substage 1: Let $n<s$ be least unsolved pair $(F_n,y_n)$ requiring attention (if such a pair does not exist then go to the next substage). Let $w_k$ be the least inactive dummy, and set $x \prec_X^{(s+1)} w_k \prec_X^{(s+1)} y_n$ for each $x \in F_n$. Declare $w_k$ active and $(F_n,y_n)$ solved.
\item
Substage 2: If $w_k$ is the least dummy symbol that is active but not replaced, and $w_k$ solves $(F_n, y_n)$, then we can assume $f^{(s)}$ is already defined for all elements in $F_n \cup \{y_n\}$, so we search for $z\in Y$ with $f(x) \prec_Y^{(s+1)} z \prec_Y^{(s+1)} f(y_n)$ for each $x \in F_n$. If such $z$ is found, then declare $w_k$ to be replaced and extend $f^{(s)}$ to $f^{(s+1)}$ by adding $2k+1$ to its domain and defining $f^{(s+1)}(2k+1)=z$. If no such $z$ is found then set $f^{(s+1)}=f^{(s)}$.
\item
Substage 3: Further extend $\prec_X^{(s+1)}$ so that $x \prec_X^{(s+1)} x'$ whenever $f(x)$ and $f(x')$ are defined and $f(x) \prec_Y^{(s+1)} f(x')$.
\item
Substage 4: Complete $\prec_X^{(s+1)}$ by adding  $\prec_X^{(s)}$ and taking the transitive closure. Then go to stage $s+1$.
\end{itemize}
Set
\[X=\{ 2n \mid n\in\N\} \cup \{ 2n+1 \mid w_n \text{ becomes active at some stage}\}.\]
Define $\prec_X = \bigcup_{s\in\N}\prec_X^{(s)}$ and $f = \bigcup_{s\in\N} f^{(s)}$. It is clear that $X$ and $\prec_X$ are c.e. Furthermore, Substage~4 guarantees that $\prec_X$ is transitive, and Substage~1 guarantees that it is interpolable.

Next, consider the case that $\prec_Y$ is interpolable. Then $f\colon X\to Y$ is a total function because each activated dummy symbol is eventually replaced in Substage~2. It is also clear that $f$ is computable, and that it is a surjection because $f(2n)=n$. Furthermore, Substage~3 guarantees that $f(x)\prec_Y f(x')$ implies $x \prec_X x'$. We show that conversely, if $x \prec_X x'$ then $f(x) \prec_Y f(x')$. We consider the different cases that $x \prec_X x'$ was added to the relation $\prec_X$ at some stage:
\begin{enumerate}
\item
If $x \prec_X x'$ was added in Substage~1, then $x$ (or $x'$) is a dummy symbol $w_k$ and $x'$ (or $x$) is from some pair $(F_n, y_n)$. Then by assumption $w_k$ is replaced in some later stage in Substage~2, and the replacement is done in such a way that $f(x) \prec_Y f(x')$ holds.
\item
If $x \prec_X x'$ was added in Substage~3 then trivially $f(x) \prec_Y f(x')$.
\item
The last possibility is that $x \prec_X x'$ was added in Substage~4 when completing under transitive closure. Then there is some $z$ such that $x \prec_X z$ and $z \prec_X x'$ were already included in $\prec_X$ before $x \prec_X x'$ was added. By induction, we can assume that $f(x) \prec_Y f(z)$ and $f(z) \prec_Y f(x')$. Therefore, $f(x) \prec_Y f(x')$ by the transitivity of $\prec_Y$.
\end{enumerate}
It follows from Lemma \ref{lem:relation-map-to-homeo} that if $\prec_Y$ is interpolable then $\I{\prec_X}$ and $\I{\prec_Y}$ are computably homeomorphic.

From the above construction, we can enumerate a sequence of c.e. interpolable transitive relations $(\prec_{X_i})_{i\in\N}$ in such a way that for every c.e. interpolable transitive relation $\prec_Y$ there is some $i\in\N$ with  $\I{\prec_{X_i}}$ computably homeomorphic to  $\I{\prec_Y}$.
\end{proof}

%%%%%%%%%%%%%%%%%%%%%%%%%%%%%%%%%%%%%%%%%%%%%%%%%%%

\section{Complexity of (effective) homeomorphism}\label{comisom}

Here we estimate the complexity of (effective) homeomorphism relations $\simeq_e$ and $\simeq$ in the introduced numberings and deduce some corollaries. Similar questions for algebraic structures were  studied in detail (see e.g. \cite{gn,ffh12,ffn12}). In the next theorem we collect some estimates which for the classes of domains resemble the corresponding estimates for algebraic structures\footnote{We thank Nikolay Bazhenov for the related bibliographical hints.}, while for Polish and quasi-Polish spaces are apparently higher.

To obtain the estimate for $\iota$, we employ the representation of computable functions $f:\I{\prec_1}\to \I{\prec_2}$ between spaces of ideals, where $\prec_1,{\prec_2}\in\mathbf{T}$, established in \cite{br2}, Theorem 2 (see the end of Section \ref{func}). 
%We associate with any $R\subseteq\mathbb{N}\times\mathbb{N}$ a partial function $\lceilx R\rceilx$ from $\I{\prec_1}$ to $\I{\prec_2}$ as follows: $\lceilx R\rceilx(I)=\{n\in\mathbb{N}\mid\exists m\in I(mRn)\}$, $dom(\lceilx R\rceilx)=\{I\in \I{\prec_1}\mid\lceilx R\rceilx(I)\in \I{\prec_2}\}$. By Theorem 2 in \cite{br2}, a function $f:\I{\prec_1}\to \I{\prec_2}$ is computable iff $f=\lceilx R\rceilx$ for some c.e.  binary relation $R$.

For the case $\prec_1,{\prec_2}\in\mathbf{I}$ of domains, the above representation may be simplified using the effective version of results in Section 2.2.6 of \cite{aj} (see Theorem \ref{categ} above). Namely, the  computable functions $f:\I{\prec_1}\to \I{\prec_2}$ coincide with the functions $\lceilx R\rceilx$ where $R$ is a binary c.e. relation on $\omega$ satisfying the following conditions: if $aR b$ and $a\prec_1a'$ then $a'R b$; if $aR b$ and $b'\prec_2b$ then $aR b'$; for any $a$ there is $b$ with  $aR b$; for all $a,b,b'$ with $aR b,aR b'$ there is $b''$ with $b\prec_2b'',b'\prec_2b''$, and $aR b''$; if $aR b$ then $a'R b$ for some $a'\prec_1a$.  Conjunction of these conditions is denoted as $mor(R,\prec_1,\prec_2)$ (meaning ``$R$ is a morphism from $\prec_1$ to $\prec_2$''). We note that $\lceilx R\rceilx=id_{\I{\prec_1}}$ iff $aRb\leftrightarrow b\prec_1a$.  

\begin{theorem}\label{indexhom}
\begin{enumerate}
 \item Let $\nu\in\{\alpha,\beta,\gamma,\delta,\varepsilon\}$. Then the relations $\nu(m)\simeq_e\nu(n)$ and $\nu(m)\simeq\nu(n)$ are $\Sigma^0_3$-complete and $\Sigma^1_1$-complete sets, respectively. Moreover, they are resp.  $\Sigma^0_3$- and $\Sigma^1_1$-complete equivalence relations under the computable reducibility of equivalence relations.
  \item Let $\nu\in\{\iota,\mu\}$. The relations $\nu(m)\simeq_e\nu(n)$ and $\nu(m)\simeq\nu(n)$ are $\Pi^1_1$ and $\Sigma^1_2$, respectively.
\end{enumerate}
 \end{theorem}
 
\begin{proof} 1. First we prove the upper bounds. For $\nu=\alpha$, it is easy to see (cf. proof of Theorem 2 in \cite{s20}) that $\alpha(m)\simeq_e\alpha(n)$ iff $(\om;V_{p(m)})\simeq_e(\om;V_{p(n)})$ iff
 \begin{eqnarray*}
 \exists k,l\forall x,y(\varphi_k(x)\downarrow\wedge\varphi_l(x)\downarrow\wedge(xV_{p(m)}y\leftrightarrow\varphi_k(x)V_{p(n)}\varphi_k(y))\wedge (xV_{p(n)}y\leftrightarrow\\\varphi_l(x)V_{p(m)}\varphi_l(y))\wedge xV_{p(m)}\varphi_l(\varphi_k(x))V_{p(m)}x\wedge yV_{p(n)}\varphi_k(\varphi_l(y)V_{p(n)}y)), 
 \end{eqnarray*}
 hence the relation is $\Sigma^0_3$. For the relation $\simeq$ we only have to add the functional quantifier $\exists h$ in the beginning of the above formula and relativize $\varphi$ to the oracle $h$; this yields the desired estimate $\Sigma^1_1$.

The above argument works for $\nu=\beta$ if we just replace $p$ by $o$. For $\nu=\gamma$, we also replace $p$ by $c$; it is easy to see that the $\emptyset''$-computability of $c$ does not damage the estimate $\Sigma^0_3$ and (trivially) the estimate $\Sigma^1_1$.

For $\nu=\delta$, we use the representation of computable functions between ideal spaces described before the formulation of the theorem: $\delta(m)\simeq_e\delta(n)$ iff 
 \begin{eqnarray*}
 \exists k,l,x,y(x=i(m)\wedge y=i(n)\wedge mor(V_k,V_x,V_y)\wedge mor(V_l,V_y,V_x)\wedge\\ \forall a,b(a(V_l\circ V_k)b\leftrightarrow bV_xa)\wedge \forall a,b(a(V_k\circ V_l)b\leftrightarrow bV_ya).
 \end{eqnarray*}
 Since $i$ is $\emptyset''$-computable, the first two conjuncts in the main parenthesis are $\Sigma^0_3$. Since $V_x$ is c.e., the same holds for the third and fourth conjuncts, while the fifth and sixth conjuncts are $\Pi^0_2$. This concludes the estimate for $\simeq_e$. For the relation $\simeq$ we only have to add the functional quantifier $\exists h$ in the beginning of the above formula and replace $V_k,V_l$ by $V_k^h,V_l^h$; this yields the desired estimate $\Sigma^1_1$.
The above argument (with $j$ in place of $i$) works for $\nu=\varepsilon$.

Now we prove the lower bounds. By Theorem 4.7(a) in \cite{gn}, for any $\Sigma^0_3$ set $A$  there are computable sequences $\{L_k\},\{M_k\}$ of computable linear orders on $\om$ such that $k\in A$ iff $L_k\simeq_eM_k$. By the definition of ideal spaces, $L_k\simeq_eM_k$ iff $I(L_k)\simeq_eI(M_k)$ iff $A\leq_m\{\langle k,l\rangle\mid\nu(k)\simeq_e\nu(l)\}$ for every $\nu\in\{\alpha,\beta,\gamma,\delta,\varepsilon\}$, concluding the proof for $\Sigma^0_3$.   

By Theorem 4.4(d) in \cite{gn}, for any $\Sigma^1_1$ set $A$  there are computable sequences $\{L_k\},\{M_k\}$ of computable linear orders on $\om$ such that $k\in A$ iff $L_k\simeq M_k$. Repeating the argument of the previous paragraph, we obtain the proof for $\Sigma^1_1$.

It remains to show that $\simeq_e$ and $\simeq$ are also complete as equivalence relations. As follows from Proposition 4 in \cite{ffn12}, for any $\Sigma^0_3$ equivalence relation $A$ on $\om$ there is a computable sequence $\{L_k\}$ of computable partial orders on $\om$ such that $kAl$ iff $L_k\simeq_eL_l$ which proves the $\Sigma^0_3$ completeness for every $\nu\in\{\alpha,\beta,\gamma,\delta,\varepsilon\}$. By Theorem 5 in \cite{ffh12}, for any $\Sigma^1_1$ equivalence relation $A$ on $\om$ there is a computable sequence $\{L_k\}$ of computable linear orders on $\om$ such that $kAl$ iff $L_k\simeq L_l$. This proves the $\Sigma^1_1$ completeness for every $\nu\in\{\alpha,\beta,\gamma,\delta,\varepsilon\}$.
 
2. By Proposition \ref{numb_compare}, we can use  $\iota$ instead of $\pi$. Denoting the relation $V_{t(n)}$ in Proposition \ref{numb_rel} by $\prec_n$, we obtain: $\iota(m)\simeq_e\iota(n)$ iff $\I{\prec_m}\simeq_e \I{\prec_n}$ iff
 \begin{eqnarray*}
 \exists k,l({\lceilx} V_k{\rceilx}:\I{\prec_m}\to \I{\prec_n}\wedge{\lceilx} V_l{\rceilx}:\I{\prec_n}\to \I{\prec_m}\wedge\\{\lceilx} V_l{\rceilx}\circ{\lceilx} V_k{\rceilx}=id_{\I{\prec_m}}\wedge{\lceilx} V_k{\rceilx}\circ{\lceilx} V_l{\rceilx}=id_{\I{\prec_n}}), 
 \end{eqnarray*}
 hence it suffices to check that the relation ${\lceilx} V_l{\rceilx}\circ{\lceilx} V_k{\rceilx}=id_{\I{\prec_m}}$ is $\Pi^1_1$. Since it is equivalent to $\forall I\in \I{\prec_m}({\lceilx} V_j{\rceilx}({\lceilx} V_i{\rceilx}(I))=I)$, this follows from the definition of ${\lceilx} R{\rceilx}(I)$.

The second assertion is a straightforward relativization of the first one. Indeed, $\iota(m)\simeq\iota(n)$ iff $\I{\prec_m}\simeq \I{\prec_n}$ iff
  \begin{eqnarray*}
  \exists R,S\subseteq\mathbb{N}^2({\lceilx} R{\rceilx}:\I{\prec_m}\to \I{\prec_n}\wedge{\lceilx} S{\rceilx}:\I{\prec_n}\to \I{\prec_m}\wedge\\{\lceilx} S{\rceilx}\circ{\lceilx} R{\rceilx}=id_{\I{\prec_m}}\wedge{\lceilx} R{\rceilx}\circ{\lceilx} S{\rceilx}=id_{\I{\prec_n}}), 
 \end{eqnarray*}
 hence the relation is $\Sigma^1_2$.
 \end{proof}

We do not currently know whether the estimates in item 2 of the above theorem are precise. From the effective Stone duality developed in \cite{HTMN,hks} it follows that the homeomorphism relation between computable compact Polish spaces is $\Sigma^1_1$-complete, as it was noticed in a recent communication of the third author with Alexander Melnikov (see Corollary 4.28 in \cite{dm}). But for computable Polish spaces the question remains open.

As a corollary of Theorem \ref{indexhom} and Proposition \ref{numb_compare}, we obtain upper bounds for $\iota$-index sets of some natural classes of spaces. 

\begin{corollary}\label{indexdom}
Let $\nu\in\{\mu,\alpha,\beta,\gamma,\delta,\varepsilon\}$. Then $\{n\mid\exists m(\iota(n)\simeq_e\nu(m))\}$ is $\Pi^1_1$ and $\{n\mid\exists m(\iota(n)\simeq\nu(m))\}$ is $\Sigma^1_2$.
 \end{corollary}
 
In particular, the problem of deciding whether a given effective quasi-Polish space is effectively homeomorphic to a metrizable space (a c.e.~domain, c.e.~algebraic domain, etc.)~is $\Pi^1_1$.
For the homeomorphism problem, it is $\Sigma^1_2$.
In the next section we show that the estimate for metrizable spaces can be improved.

\section{Complexity of separation axioms}\label{complete}

Here we discuss some classes of spaces related to separation axioms. Let $\mathcal{T}_1,\mathcal{T}_2,\mathcal{R},\mathcal{M}$ be the classes of $T_1$-, $T_2$-, regular, and metrisable spaces, respectively. Let $\{D_n\}$ be the standard numbering of finite subsets of $\omega$, then the sets $\check{D}_n=\{A\subseteq\omega\mid D_n\subseteq A\}$ form the standard basis of the Scott topology on $P\omega$. 

\begin{proposition}\label{index}
The $\pi$-index set of any of the classes $\mathcal{T}_1,$ $\mathcal{T}_2,$ $\mathcal{R},$ $\mathcal{M}$ is $\Pi^1_1$.
 \end{proposition}

\begin{proof} By the definition of a $T_1$-space, $\pi(m)\in\mathcal{T}_1$ iff
 $\forall x,y\in\pi(m)(x\neq y\to\exists n (x\in\check{D}_n\not\ni y)).$
 Since $\pi(m)\in\Pi^0_2(P\omega)$, we get $\pi^{-1}(\mathcal{T}_1)\in\Pi^1_1$.

By the definition of a $T_2$-space, $\pi(m)\in\mathcal{T}_2$ iff
 $$\forall x,y\in\pi(m)(x\neq y\to\exists i,j( x\in\check{D}_i\wedge y\in\check{D}_j\wedge\check{D_i}\cap\check{D}_j\cap\pi(m)=\emptyset)).$$
  Since $\check{D}_i\cap\check{D}_j\cap\pi(m)=\emptyset$ iff $\forall z\in\pi(m)(z\not\in\check{D}_i\vee z\not\in\check{D}_j)$, we have $\pi^{-1}(\mathcal{T}_2)\in\Pi^1_1$.
  
Recall that $X$ is regular iff for every $x\in X$ and every basic neighborhood $U$ of $x$ there is a basic neighborhood $V$ of $x$ such that the closure $Cl(V)$ of $V$ in $X$ is contained in $U$. For $X=\pi(m)$ this reads: $\pi(m)\in\mathcal{R}$ iff
 $\forall x\in\pi(m)\forall i(x\in\check{D}_i\to\exists j (x\in\check{D}_j\wedge Cl(\check{D}_j\cap\pi(m))\subseteq \check{D}_i)).$ Thus, it suffices to check that the relation $\forall y(y\in Cl(\check{D}_j\cap\pi(m))\to y\in\check{D}_i)$ is $\Pi^1_1$, and for this it suffices to check that the relation $y\in Cl(\check{D}_j\cap\pi(m))$ is $\Sigma^1_1$. The relation is equivalent to
 $\forall k(y\in \check{D}_k\cap\pi(m)\to\exists z\in \check{D}_j\cap\pi(m)(z\in\check{D}_k)),$ hence it is indeed $\Sigma^1_1$.

By the Urysohn metrisation theorem we have $\mathcal{M}=\mathcal{T}_1\cap\mathcal{R}$, hence the  estimate $\Pi^1_1$ for $\pi^{-1}(\mathcal{M})$ follows from the previous  ones. Note that the upper bound $\Sigma^1_2$ of $\pi^{-1}(\mathcal{M})$ in Corollary \ref{indexdom} (without using the Urysohn  theorem) is much worse. 
 \end{proof}

Next we show that the upper bounds of Proposition \ref{index} are optimal. Our proofs below demonstrate that the ideal characterisations  provide  useful tools for such kind of results. Recall that the following implications hold for cb$_0$-spaces:
\[\mbox{metrizable}\iff\mbox{regular}\implies\mbox{Hausdorff}\implies T_1.\]

For the equivalence of metrizability and regularity, as mentioned in \cite[Page 12]{CiT}, every regular $T_0$ space is Hausdorff:
For two distinct points $x,y$, by $T_0$-ness, there exists an open set $U$ containing either $x$ or $y$, but not the other.
Assuming that $U$ contains $x$, the complement of $U$ is closed, so from regularity they are separated, which in particular separates $x$ and $y$.
Thus, it follows from the Urysohn metrization theorem that regular cb$_0$ space is metrizable.

We start with the following $\Pi^1_1$-completeness result with respect to the numbering $\iota$ (i.e., the numbering of all effective quasi-Polish spaces induced from the standard numbering of c.e.~transitive relations), where recall $\iota\equiv_e\pi$ from Proposition \ref{numb_compare}.

\begin{theorem}\label{prop:Pi11completeness-metrizable}
Let $F\subseteq\om$ be a $\Pi^1_1$ set.
Then, there exists a computable function which, given $p\in\om$, returns an $\iota$-index of a c.e. EQP-space $X$ such that
\[
\begin{cases}
\mbox{$X$ is metrizable} & \mbox{if }p\in F,\\
\mbox{$X$ is not $T_1$} & \mbox{if }p\not\in F.
\end{cases}
\]
\end{theorem}

\begin{proof}
Recall that the set of indices of well-founded computable trees is $\Pi^1_1$ complete.
Hence, instead of a $\Pi^1_1$ set, we consider computable trees.
Let $T\subseteq\om^{<\om}$ be a computable tree.
Our space $X$ will be $\{I_x:x\in\om^\om\}\cup\{J_x:x\in [T]\}$ equipped with the specialization order $J_x\leq I_x$, where $[T]$ is the set of all infinite paths through $T$.
The discussion from here on is to write down this space $X$ as an ideal space.

For each $\sigma\in\om^{<\om}$, we prepare for a new symbol $\underline{\sigma}$.
Let $|\sigma|$ be the length of $\sigma$, and put $|\underline{\sigma}|=|\sigma|$.
If $\sigma$ is nonempty, i.e., $|\sigma|>0$, we denote by $\sigma^-$ the immediate predecessor of $\sigma$.
We define a computable binary relation $\prec$ on the set $|\mbox{$\prec$}|:=\{\sigma,\underline{\sigma}:\sigma\in\om^{<\om}\}$ as follows:
If $\sigma\in\om^\om$ is nonempty, enumerate $\sigma^-\prec\sigma$, $\underline{\sigma^-}\prec\underline{\sigma}$, and $\sigma\prec\underline{\sigma}$.
If $\sigma\not\in T$ then we also enumerate $\underline{\sigma^-}\prec\sigma$.
Then consider its transitive closure and define $X={I}(\prec)$.

Note that if $a\prec b$ then either $|a|<|b|$ or $a=\sigma$ and $b=\underline{\sigma}$ for some $\sigma\in\om^{<\om}$.
If $I$ is an ideal of $\prec$, then for any $a\in I$ one can use directedness of $I$ twice to obtain $b,c\in I$ such that $a\prec b\prec c$.
Then, by the property of $\prec$ mentioned above, we have $|a|<|c|$.
Therefore, any ideal contains arbitrarily long strings.
Moreover, as no pair of incomparable strings has an upper bound, in order for a set to be directed, all of its members must be comparable.
This means that for any ideal $I$ of $\prec$ there exists an infinite string $x\in\om^\om$ such that $I$ consists only of the initial segments of $x$ or those underlined in them.
In other words, $I$ is the $\prec$-downward closure of $\{\sigma:\sigma\subset x\}$ or $\{\underline{\sigma}:\sigma\subset x\}$, where we mean by $\sigma\subset x$ that $\sigma$ is an initial segment of $x$.

If $x$ is an infinite path through $T$, then the downward closure of $\{\sigma:\sigma\subset x\}$ is $J_x=\{\sigma:\sigma\subset x\}$, and the downward closure of $\{\underline{\sigma}:\sigma\subset x\}$ is $I_x=\{\sigma,\underline{\sigma}:\sigma\subset x\}$.
Both $I_x$ and $J_x$ are ideals, and since $J_x\subseteq I_x$, obviously $J_x\in[\tau]_\prec$ implies $I_x\in[\tau]_\prec$, so $J_x\leq_XI_x$, where recall that $[n]_\prec=\{I\in \I{\prec}:n\in I\}$ is a basic open set, and $\leq_X$ is the specialization order.
Hence, if $T$ is not well-founded, then $X$ is not $T_1$.
If $T$ is well-founded, then any $x\in\om^\om$ has an initial segment $\sigma\not\in T$, and for any such $\sigma$ we have $\sigma^-\prec\underline{\sigma^-}\prec\sigma\prec\underline{\sigma}\prec\dots$.
Hence, $\{\sigma:\sigma\subset x\}$ and $\{\underline{\sigma}:\sigma\subset x\}$ have the same downward closure $I_x=\{\sigma,\underline{\sigma}:\sigma\subset x\}$.
Therefore, any ideal is of the form $I_x$ for some $x\in\om^\om$.
Thus, we have $X=\{I_x:x\in\om^\om\}$, which is homeomorphic to Baire space $\om^\om$.
This is because, as $|\mbox{$\prec$}|=\{\sigma,\underline{\sigma}:\sigma\in\om^{<\om}\}$ is the underlying set of the binary relation $\prec$, the set $\{[\sigma]_\prec,[\underline{\sigma}]_\prec:\sigma\in\om^{<\om}\}$ yields the topology on $X$ by definition, and the above argument shows $X\cap[\sigma]_{\prec}=X\cap[\underline{\sigma}]_\prec=\{I_x:\sigma\subset x\}$.
In particular, $X$ is metrizable.
 
For overtness, given $\sigma\in\om^{<\om},$ if $x$ extends $\sigma$ then $I_x=\{\sigma,\underline{\sigma}:\sigma\subset x\}$ is an ideal of $\prec$ as seen above, and contains both $\sigma$ and $\underline{\sigma}$; hence $I_x\in[\sigma]_\prec$ and $I_x\in[\underline{\sigma}]_\prec$.
This means that $X\cap[\tau]_\prec\not=\emptyset$ for any $\tau\in|\mbox{$\prec$}|$.
In particular, $X$ is overt.
\end{proof}

Theorem \ref{prop:Pi11completeness-metrizable} shows that, for any $i\in\{1,2,3\}$, the $\iota$-index set of all  c.e. EQP $T_i$-spaces is $\Pi^1_1$-complete, where a second countable $T_0$ space is $T_3$ if and only if it is metrizable.
This result can be further extended as follows.

\begin{theorem}\label{prop:Pi11completeness-general}
Let $M\subseteq H\subseteq F\subseteq\om$ be $\Pi^1_1$ sets.
Then, there exists a computable function which, given $p\in\om$, returns an $\iota$-index of a c.e. EQP-space $X$ such that
\[
\begin{cases}
\mbox{$X$ is metrizable} & \mbox{if }p\in M,\\
\mbox{$X$ is Hausdorff, but not metrizable} & \mbox{if }p\in H\setminus M,\\
\mbox{$X$ is $T_1$, but not Hausdorff} & \mbox{if }p\in F\setminus H,\\
\mbox{$X$ is not $T_1$} & \mbox{if }p\not\in F.
\end{cases}
\]
\end{theorem}

This means that every tuple $(M,H,F)$ of $\Pi^1_1$-sets such that $M\subseteq H\subseteq F$ uniformly $m$-reduces to $( \iota^{-1}(\mathcal{M}),\iota^{-1}(\mathcal{T}_2),\iota^{-1}(\mathcal{T}_1))$. 
Let us decompose the proof of Theorem \ref{prop:Pi11completeness-general} into a few lemmas.

\begin{lemma}\label{prop:T2non-metrizable}
Let $H\subseteq\om$ be a $\Pi^1_1$ set.
Then, there exists a computable function which, given $p\in\om$, returns an $\iota$-index of a c.e. EQP-space $X$ such that
\[
\begin{cases}
\mbox{$X$ is metrizable} & \mbox{if }p\in H,\\
\mbox{$X$ is $T_1$, but not Hausdorff} & \mbox{if }p\not\in H.
\end{cases}
\]
\end{lemma}

\begin{proof}
First, one specific example of a second countable $T_1$ topology which is not Hausdorff is called a telophase topology \cite[II.73]{CiT}.
Here, our construction is closer to the one in \cite{KNP}, which adds an inseparable pair of points at infinity to $\om$ than the one in \cite[II.73]{CiT}, which adds a new point $1^\star$ to $[0,1]$ where $(1,1^\star)$ forms an inseparable pair.
In our construction, a tree $T\subseteq\om^{<\om}$ is first given.
For $x\in\om^\om$, if $x$ is an infinite path through $T$ then we add an inseparable pair $(I_x,I_x^\star)$ of points at infinity to the discrete space $\om^{<\om}$.
If $x$ is not an infinite path through $T$ then we add a single point $J_x$ at infinity to $\om^{<\om}$.

Formally, given a tree $T\subseteq\om^{<\om}$, we consider the following specific presentation $\prec$ of a telophase topology:
%For $n\in\om$, prepare for symbols $\underline{n}$, $[n,\infty]$, and $[n,\infty^\star]$.
%Then, enumerate $[m,\om]\prec_1[n,\om]\prec_1\underline{n}\prec_1\underline{n}$ for any $m\leq n$ and $\om\in\{\infty,\infty^\star\}$.
%Then ${\rm Id}(\prec_1)$ is $T_1$ but not $T_2$.
For each $\sigma\in\om^{<\om}$, we prepare for symbols $\underline{\sigma}$, $[\sigma,\infty]$, and $[\sigma,\infty^\star]$.
We define a computable binary relation $\prec$ on the set $|\mbox{$\prec$}|:=\{\underline{\sigma},[\sigma,\infty],[\sigma,\infty^\star]:\sigma\in\om^{<\om}\}$.
If $\sigma$ is nonempty, we denote by $\sigma^-$ the immediate predecessor of $\sigma$, and enumerate $[\sigma^-,o]\prec[\sigma,o]\prec\underline{\sigma}\prec\underline{\sigma}$ for each $o\in\{\infty,\infty^\star\}$.
If $\sigma\not\in T$, we also enumerate $[\sigma^-,\infty^\star]\prec[\sigma,\infty]\prec[\sigma,\infty^\star]$.
Then consider its transitive closure and define $X=\I{\prec}$.

First, since $\underline{\sigma}\prec\underline{\sigma}$, the $\prec$-downward closure of $\{\underline{\sigma}\}$ forms an ideal.
This is $I_\sigma=\{\underline{\sigma}\}\cup\{[\tau,\infty],[\tau,\infty^\star]:\tau\subseteq\sigma\}$, where we mean by $\tau\subseteq\sigma$ that $\tau$ is an initial segment of $\sigma$.
Note that the subspace $Y=\{I_\sigma:\sigma\in\om^{<\om}\}$ of $X$ is discrete since $Y\cap [\underline{\sigma}]_\prec=\{I_\sigma\}$.
For any $a\not\in I_\sigma$, $a$ and $\underline{\sigma}$ have no common upper bound, so $I_\sigma$ is the unique ideal containing $\underline{\sigma}$.
If an ideal $I$ does not contain $\underline{\sigma}$ for any $\sigma\in\om^{<\om}$, then as in the proof of Theorem \ref{prop:Pi11completeness-metrizable}, one can see that $I$ contains $[\sigma,\infty]$ or $[\sigma,\infty^\ast]$ for an arbitrarily long string $\sigma$.
Moreover, as no pair of incomparable strings has an upper bound, in order for a set to be directed, all of its members must be comparable.
This means that for any such ideal $I$ of $\prec$ there exists an infinite string $x\in\om^\om$ such that $I$ consists only of $[\sigma,\infty]$ or $[\sigma,\infty^\star]$ for initial segments $\sigma$ of $x$.
In other words, such an $I$ is the $\prec$-downward closure of $\{[\sigma,\infty]:\sigma\subset x\}$ or $\{[\sigma,\infty^\star]:\sigma\subset x\}$.

If $x$ is an infinite path through $T$, then both $I_x=\{[\sigma,\infty]:\sigma\subset x\}$ and $I_x^\star=\{[\sigma,\infty^\star]:\sigma\subset x\}$ are downward closed.
Hence, any ideal is of the form $I_\sigma$, $I_x$ or $I_x^\star$.
We claim that the latter two ideals as points cannot be separated by disjoint open sets.
This is because any basic open sets containing $I_x$ and $I_x^\star$ are of the form $[[\sigma,\infty]]_\prec$ for some $\sigma\subset x$ and $[[\tau,\infty^\star]]_\prec$ for some $\tau\subset x$ respectively.
However, $[[\sigma,\infty]]_\prec$ and $[[\tau,\infty^\star]]_\prec$ always have an intersection $I_\rho$, where $\rho$ is a common extension of $\sigma$ and $\tau$.
Hence, $X$ is not Hausdorff.
If $x$ is not an infinite path through $T$, then as in the proof of Theorem \ref{prop:Pi11completeness-general}, one can see that both $\{[\sigma,\infty]:\sigma\subset x\}$ and $\{[\sigma,\infty^\star]:\sigma\subset x\}$ have the same downward closure $J_x=\{[\sigma,\infty],[\sigma,\infty^\star]:\sigma\subset x\}$.
In any case, no two ideals are comparable by $\subseteq$, so $X$ is $T_1$.
Hence, if $T$ is ill-founded, then $X$ is $T_1$, but not Hausdorff.

If $T$ is well-founded, then as seen above, any ideal is of the form $I_\sigma$ or $J_x$; that is, $X=\{I_\sigma:\sigma\in\om^{<\om}\}\cup\{J_x:x\in\om^\om\}$.
We claim that $X$ is homeomorphic to the Polish space $\om^{\leq\om}:=\om^{<\om}\cup\om^\om$ whose topology is generated from $\{\sigma:\sigma\in\om^{<\om}\}$ and $[\sigma]=\{x\in\om^{\leq\om}:x\mbox{ extends }\sigma\}$.
This is because, as $|\mbox{$\prec$}|=\{\underline{\sigma},[\sigma,\infty],[\sigma,\infty^\star]:\sigma\in\om^{<\om}\}$ is the underlying set of the binary relation $\prec$, the set $\{[\underline{\sigma}]_\prec,[[\sigma,\infty]]_\prec,[[\sigma,\infty^\star]]_\prec:\sigma\in\om^{<\om}\}$ yields the topology on $X$ by definition, and the above argument shows $X\cap[[\sigma,\infty]]_\prec=X\cap[[\sigma,\infty^\star]]_\prec=\{I_\tau:\tau\subseteq\sigma\}\cup\{J_x:\sigma\subset x\}$.
Hence, the union of the map $\sigma\mapsto I_\sigma$ and the map $x\mapsto J_x$ gives a homeomorphism between $\om^{\leq\om}$ and $X$.
In particular, $X$ is metrizable.

For overtness, given $\sigma\in\om^{<\om}$, we have $I_\sigma\in[\sigma]_\prec$, and if $x$ extends $\sigma$ then $I_x,J_x\in[[\sigma,\infty]]_\prec$ and $I_x^\star,J_x\in[[\sigma,\infty^\star]]_\prec$.
This means that $X\cap[\tau]_\prec\not=\emptyset$ for any $\tau\in|\mbox{$\prec$}|$.
In particular, $X$ is overt.
\end{proof}

\begin{lemma}\label{prop:T3non-metrizable}
Let $M\subseteq\om$ be a $\Pi^1_1$ set.
Then, there exists a computable function which, given $p\in\om$, returns an $\iota$-index of a c.e. EQP-space $X$ such that
\[
\begin{cases}
\mbox{$X$ is metrizable} & \mbox{if }p\in M,\\
\mbox{$X$ is Hausdorff, but not metrizable} & \mbox{if }p\not\in M.
\end{cases}
\]
\end{lemma}

\begin{proof}
First, one specific example of a second countable Hausdorff topology which is not metrizable is called a double origin topology \cite[II.74]{CiT}.
It is like a Euclidean plane with two origins, which cannot be separated by closed neighborhoods (that cause non-metrizability).
Here, our construction is closer to the one in \cite{KNP}, which yields a quasi-Polish space, while the example in \cite[II.74]{CiT} is not quasi-Polish.
In our construction, a tree $T\subseteq\om^{<\om}$ is first given.
The base plane of our space is the discrete space $\om\times(\om^{<\om}\sqcup\om^{<\om})$.
For each $x\in\om^\om$, the points $I_{n,x}=(n,x)$ and $(\infty,x)$ may be added.
Here, if $x$ is an infinite path through $T$ then two points $J_x^+$ and $J_x^-$ corresponding to $(\infty,x)$ are added, and these cannot be separated by closed neighborhoods.
Indeed, the intersection of any two closed neighborhoods containing $J_x^+$ and $J_x^-$ respectively contains $I_{n,x}$ for an arbitrary large $n$.
If $x$ is not an infinite path through $T$ then the plane is folded in half with the abscissa $\{I_{n,x}:n\in\om\}$ as the fold line, and then $J_x^+$ is identified with $J_x^-$.

Formally, given a tree $T\subseteq\om^{<\om}$, we consider the following specific presentation $\prec$ of a telophase topology:
For each $n\in\om$ and $\sigma\in\om^{<\om}$, we prepare for symbols $(n,\sigma)$, $(n,\underline{\sigma})$, $(n,\sigma^{\pm})$, $[n,\sigma]$ and $[n,\underline{\sigma}]$.
We define a computable binary relation $\prec$ on the set 
\[|\mbox{$\prec$}|:=\{(n,\sigma), (n,\underline{\sigma}), (n,\sigma^{\pm}), [n,\sigma],[n,\underline{\sigma}]:n\in\om\mbox{ and }\sigma\in\om^{<\om}\}.\]

For any $m<n$, put the following:
\begin{gather*}
[m,\sigma]\prec[n,\sigma]\prec (n,\sigma)\prec (n,\sigma),\\
[m,\underline{\sigma}]\prec[n,\underline{\sigma}]\prec (n,\underline{\sigma})\prec (n,\underline{\sigma}).
\end{gather*}

If $\tau$ is a proper initial segment of $\sigma$, put the following:
\[(n,\tau^\pm)\prec(n,\sigma^\pm),\quad (n,\sigma^\pm)\prec (n,\sigma),\quad (n,\sigma^\pm)\prec (n,\underline{\sigma}).\]

If $m<n$ and $\tau$ is a proper initial segment of $\sigma$, put the following:
\[[m,\tau]\prec[n,\sigma],\quad [m,\underline{\tau}]\prec[n,\underline{\sigma}].\]

If $\sigma\not\in T$, $m<n$, and $\tau$ is a proper initial segment of $\sigma$, then we also put the following:
\[[m,\tau]\prec[n,\underline{\sigma}]\prec[n,\sigma],\quad (n,\sigma)\prec(n,\underline{\sigma})\prec(n,\sigma).\]

Then consider its transitive closure and define $X=\I{\prec}$.
For a directed set $D$, let ${\downarrow}D$ denote the $\prec$-downward closure of $D$.
As in the previous proofs, one can see that any ideal of $\prec$ is one of the following forms:
\begin{gather*}
I_{n,\sigma}={\downarrow}\{(n,\sigma)\},\quad I_{n,\underline{\sigma}}={\downarrow}\{(n,\underline{\sigma})\},\quad I_{n,x}={\downarrow}\{(n,\sigma^\pm):\sigma\subset x\},\\
J_x^+={\downarrow}\{[n,\sigma]:n\in\om\mbox{ and }\sigma\subset x\},\quad J_x^-={\downarrow}\{[n,\underline{\sigma}]:n\in\om\mbox{ and }\sigma\subset x\}
\end{gather*}

If $x$ is not an infinite path through $T$, then it is easy to see that $I_{n,\sigma}=I_{n,\underline{\sigma}}$ and $J_x^+=J_x^-$.
We claim that $X$ is Hausdorff.
First, to see $I_{n,\sigma}$ and $I_{m,\tau}$ are separated for $n\not=m$ or $\sigma\not=\tau$ where $\sigma,\tau\in\om^{<\om}$, note that $(n,\sigma)$ and $(m,\tau)$ have no common upper bound, so $[(n,\sigma)]_\prec$ and $[(m,\tau)]_\prec$ have no intersection.
Thus, the points $I_{n,\sigma}$ and $I_{m,\tau}$ are separated by $[(n,\sigma)]_\prec$ and $[(m,\tau)]_{\prec}$.
Similarly, one can see that $I_{n,\underline{\sigma}}$ and $I_{n,\underline{\tau}}$ are separated.
If $\sigma\in T$ then $I_{n,\sigma}$ and $I_{n,\underline{\sigma}}$ are separated, and if $\sigma\not\in T$ then $I_{n,\sigma}=I_{n,\underline{\sigma}}$.
If $m>n$ then $(n,\sigma^\pm)$ and $[m,\sigma]$ have no common upper bound, so $[(n,\sigma^\pm)]_\prec$ and $[(m,\sigma)]_\prec$ have no intersection.
Thus, $I_{n,x}$ and $J_x^+$ are separated by them.
In a similar manner, one can easily separate pairs $(I_{n,x},J_y^-)$, $(I_{n,\sigma},J_x^+)$, $(I_{n,\sigma},I_{n,x})$, etc.
If $x$ is an infinite path through $T$, then $[n,\sigma]$ and $[n,\underline{\sigma}]$ have no common upper bound, so $[[n,\sigma]]_\prec$ and $[[n,\underline{\sigma}]]_\prec$ have no intersection.
Thus, $J_x^+$ and $J_x^-$ are separated by them.
If $x$ is not an infinite path through $T$, then $J_x^+=J_x^-$.
This concludes that $X$ is Hausdorff.

If $x$ is an infinite path through $T$, we claim that $J_x^+$ and $J_x^-$ cannot be separated by closed neighborhoods.
Indeed, we show that any closed neighborhood of $J_x^+$ or $J_x^-$ contains $I_{n,x}$ for some $n\in\om$.
To see this, consider an open neighborhood $[[n,\sigma]]_{\prec}$ of $J_x^+$.
Then, for any $m\geq n$, $[(m,\sigma^\pm)]_\prec$ is an open neighborhood of $I_{m,x}$.
Since $m\geq n$, $(m,\sigma)$ is a common upper bound of $[n,\sigma]$ and $(m,\sigma^\pm)$, so we have $I_{m,\sigma}\in[(m,\sigma)]_\prec\subseteq[[n,\sigma]]_{\prec}\cap[(m,\sigma^\pm)]_\prec$.
Hence, any open neighborhood of $I_{m,x}$ intersects with $[[n,\sigma]]_{\prec}$, and this means that the closure of $[[n,\sigma]]_\prec$ contains $I_{m,x}$ for any $m\geq n$.
Similarly, the closure of $[[n,\underline{\sigma}]]_\prec$ contains $I_{m,x}$ for any $m\geq n$.
This verifies the claim.
In particular, if $T$ is ill-founded, then such an $x$ exists, so $X$ is not metrizable.

If $T$ is well-founded, then $I_{n,\sigma}=I_{n,\underline{\sigma}}$ and $J_x^+=J_x^-$.
Hence,
\[X=\{I_{n,\sigma}:n\in\om\mbox{ and }\sigma\in\om^{<\om}\}\cup\{I_{n,x}:n\in\om\mbox{ and }x\in\om^\om\}\cup\{J_x^+:x\in\om^\om\}.\]

We claim that $X$ is embedded into the Polish space $Z=(\om+1)\times\om^{\leq\om}$, where $\om+1$ is the one point compactification of $\om$, and $\om^{\leq\om}$ endowed with the Polish topology as in the proof of Lemma \ref{prop:T2non-metrizable}.
Indeed, the union of the maps $\langle n,\sigma\rangle\mapsto I_{n,\sigma}$, $\langle n,x\rangle\mapsto I_{n,x}$, and $\langle\om,x\rangle\mapsto J_x^+$ gives a homeomorphism between $(\om\times\om^{\leq\om})\cup(\{\om\}\times\om^\om)\subseteq Z$ and $X$.
This is because we have $[(n,\sigma)]_\prec=\{I_{n,\sigma}\}$, and if $T$ has no infinite path extending $\sigma$, then we have $[(n,\sigma^\pm)]_\prec=\{I_{n,\tau}:\sigma\subseteq\tau\}\cup\{I_{n,x}:\sigma\subset x\}$, and $[[n,\sigma]]_\prec=\{I_{m,\tau}:n\leq m,\sigma\subseteq\tau\}\cup\{I_{n,x}:n\leq m,\sigma\subset x\}\cup\{J_x^+:\sigma\subset x\}$.
This means that the basic open set $[(n,\sigma)]_\prec$ in $X$ corresponds to the basic open set $\{\langle n,\sigma\rangle\}$ in $Z$, the basic open set $[(n,\sigma^\pm)]_\prec$ in $X$ corresponds to the basic open set $\{n\}\times\{x\in\om^{\leq\om}:\sigma\subset x\}$ in $Z$, and the basic open set $[[n,\sigma]]_\prec$ in $X$ corresponds to the basic open set $\{m\in\om+1:m\geq n\}\times\{x\in\om^{\leq\om}:\sigma\subset x\}$ in $Z$.
Hence, if $T$ is well-founded, then $X$ is metrizable.
Overtness of $X$ is obvious as before.
\end{proof}

\begin{proof}[Proof of Theorem \ref{prop:Pi11completeness-general}]
Let $M\subseteq H\subseteq F\subseteq\om$ be $\Pi^1_1$ sets.
Let $X_F$, $X_H$ and $X_M$ be c.e. EQP-spaces obtained by Theorem \ref{prop:Pi11completeness-metrizable}, Lemma \ref{prop:T2non-metrizable} and Lemma \ref{prop:T3non-metrizable}.
Then, consider the disjoint union of these spaces, i.e., $X=(\{0\}\times X_F)\cup(\{1\}\times X_H)\cup(\{2\}\times X_M)$.
If $p\in M$ then all of these spaces are metrizable, so $X$ is metrizable.
If $p\in H\setminus M$, then $X_F$ and $X_H$ are metrizable, and $X_M$ is Hausdorff, but not metrizable.
Therefore, $X$ is Hausdorff, but not metrizable.
If $p\in F\setminus H$, then $X_F$ is metrizable, $X_H$ is $T_1$, but not Hausdorff, and $X_M$ is Hausdorff.
Therefore, $X$ is $T_1$, but not Hausdorff.
If $p\not\in F$, then $X_F$ is not $T_1$.
Thus, $X$ is not $T_1$.
%This completes the proof.
\end{proof}

\begin{remark}
Our proof of Lemma \ref{prop:T3non-metrizable} actually gives a metrizable space if $p\in M$, and a Hausdorff but not $T_{2.5}$ space if $p\not\in M$.
On the other hand, an anonymous referee suggested an alternative proof of Lemma \ref{prop:T3non-metrizable}, which gives a metrizable space if $p\in M$, and a $T_{2.5}$ (indeed, submetrizable) but not metrizable space if $p\not\in M$:
It is the product space $\om^\om\times [0,1]^2$ with $C=[T]\times \{(x,y):x=0\text{ and }y>0\}$ added to the topology as a closed set.
Taken together, Theorem \ref{prop:Pi11completeness-general} is more complete:
Every tuple $(M,U,H,F)$ of $\Pi^1_1$-sets such that $M\subseteq U\subseteq H\subseteq F$ uniformly $m$-reduces to $(\iota^{-1}(\mathcal{M}),\iota^{-1}(\mathcal{T}_{2.5}),\iota^{-1}(\mathcal{T}_2), \iota^{-1}(\mathcal{T}_1))$. 
\end{remark}

%%%%%%%%%%%%%%%%%%%%%%%%%%%%%%%%%%%%%%%%%%%%%

%%%%%%%%%%%%%%%%%%%%%%%%%%%%%%%%%%%%%%%%%%%%%%%%%%%
%%%%%%%%%%%%%%%%%%%%%%%%%%%%%%%%%%%%%%%%%%%%%%%%%%%
%%%%%%%%%%%%%%%%%%%%%%%%%%%%%%%%%%%%%%%%%%%%%%%%%%%
%%%%%%%%%%%%%%%%%%%%%%%%%%%%%%%%%%%%%%%%%%%%%%%%%%%
%%%%%%%%%%%%%%%%%%%%%%%%%%%%%%%%%%%%%%%%%%%%%%%%%%%

\section{Degree spectra of continuous domains}\label{spectra}

An interpolable relation $\prec$ differs from a preorder in that it does not satisfy $x\prec x$ in general; interpreted in terms of domain, an element $x$ satisfying $x\prec x$ corresponds to a compact element.
Since the degree-spectra of algebraic domains have been studied to some extent \cite{s20}, let us investigate the degree spectra of non-algebraic $\omega$-continuous domains.
We say that a binary relation $\prec$ is {\em irreflexive} if $x\prec x$ fails for any $x$.
We first show that, for any Turing degree ${\bf a}$, there exists an irreflexive interpolable relation whose degree spectrum is $\{{\bf x}:{\bf a}\leq{\bf x}'\}$.

\begin{theorem}\label{thm:degree-spectra-interpolable}
For any $A\subseteq\N$, there exists an irreflexive interpolable relation $\prec_A$ on $\N$ such that for any $X\subseteq\N$:
\[A\leq_TX'\iff\mbox{there exists an $X$-computable isomorphic copy of $\prec_A$}.\]
\end{theorem}

\begin{proof}
For $\ell\in\omega+1$, let $\prec_\ell$ be a binary relation on $\ell\times\mathbb{Q}$ defined as follows:
\[(m,p)\prec_\ell(n,q)\iff m\leq n\mbox{ and }p<q.\]

Clearly, $\prec_\ell$ is irreflexive and transitive.
We claim that $\prec_\ell$ is interpolable.
To see this, let $F=\{(m_i,p_i)\}_{i<j}$ be a finite subset of $\ell\times\mathbb{Q}$, and $(n,q)$ be an $\prec_\ell$-upper bound of $F$.
Then we have $\max_{i<j}m_i\leq n$ and $\max_{i<j}p_i<q$.
By density of $\mathbb{Q}$ we have $\max_{i<j}p_i<r<q$, and so $(n,r)<(n,q)$ is also an $\prec_\ell$-upper bound of $F$.

For any set $A\subseteq\N$ we will define a binary relation $\prec_A$ on $\om^2\times\mathbb{Q}$.
In the following we fix a bijective coding $\langle\cdot,\cdot\rangle\colon\om\times 2\to\omega$, say $\langle a,b\rangle=2a+b$.
Then we define a binary relation $\prec_A$ as the disjoint union of $\{\prec_{\langle n,A(n)\rangle}:n\in\N\}$ and countably many copies of $\prec_\omega$.
To be more precise, for each $a\in\om$, if $a$ is even, say $a=2a'$,  put $\ell(a)=\langle a',A(a')\rangle$; if $a$ is odd, put $\ell(a)=\om$.
Then $\prec_A$ is defined as follows:
\[
(a,m,p)\prec_A(b,n,q)\iff a=b\mbox{ and }(m,p)\prec_{\ell(a)}(n,q).
\]

Assume that there exists an $X$-computable relation $\prec$ on $\N$ which is isomorphic to $\prec_A$.
We say that $\{x_1,\dots,x_n\}$ is {\em compatible} if $\{x_1,\dots,x_n\}$ has a $\prec$-lower bound.
We claim that for any $a\in \N$, $A(a)=b$ if and only if there exists a maximal compatible antichain of size $\langle a,b\rangle$.
To see this, note that if $h$ is an isomorphism between $\prec$ and $\prec_A$, then compatibility of $\{x_1,\dots,x_n\}$ ensures that $\{h(x_i):i\leq n\}$ belongs to a single component; that is, $\{h(x_i):i\leq n\}\subseteq\{a\}\times \ell(a)\times\mathbb{Q}$ for some $a\in \N$.
So one may assume that $h(x_i)$ is of the form $(a,m_i,p_i)$.
If such $\{x_1,\dots,x_n\}$ is an antichain, then $i\not=j$ implies $m_i\not=m_j$, so we get $n\leq\ell(a)$.
If $n<\ell(a)$, we have $k<\ell(a)$ such that $k\not=m_i$ for any nonzero $i\leq n$.
Since $\{x_1,\dots,x_n\}$ is antichain, $i<j$ implies $q_j\leq q_i$, so one can choose $q\in\mathbb{Q}$ such that $\max_{m_j>k}q_j\leq q\leq\min_{m_i<k}q_i$, and put $x'=h^{-1}(a,k,q)$.
One can see that $\{x_1,\dots,x_n,x'\}$ is still an antichain.
Hence, if $\{x_1,\dots,x_n\}$ is a maximal compatible antichain then $n=\ell(a)$.
Conversely, for a fixed $p\in\mathbb{Q}$, since $\{(a,i,p):i<\ell(a)\}$ is a maximal compatible antichain, so is $\{h^{-1}(a,i,p):i<\ell(a)\}$.
This verifies the claim.
%Consequently, for any $a\in \N$, $A(a)=b$ if and only if there exists a maximal compatible antichain of size $\langle a,b\rangle$.

We next claim that there exists an $X'$-computable sequence of all finite maximal compatible antichains.
One can easily see that the condition of being a compatible antichain for a finite sequence is c.e.~relative to $X$.
Note that a compatible antichain $\langle x_1,\dots,x_n\rangle$ is maximal if and only if any $y$ in the same component as $x_i$ is comparable with $x_j$ (i.e., $y\prec x_j$ or $x_j\prec y$) for some $j\leq n$.
This is because $h(x_i)$ is of the form $(a,m_i,p_i)$, and $h(y)$ is of the form $(a,k,q)$ for some $k<\ell(a)$.
Note that $y$ being in the same component as $x_i$ is equivalent to $\{x_i,y\}$ being compatible, so this property is $X$-c.e.
Thus, maximality of a compatible antichain is co-c.e.~relative to $X$.
This verifies the claim.

Finally, for any $a\in\N$, using $X'$ search for a maximal compatible antichain whose size is of the form $\langle a,b\rangle$.
Then, we calculate $A(a)=b$.
Hence, $A\leq_TX'$.

Conversely, assume that $A\leq_TX'$.
Let $(A_s)_{s\in\om}$ be an $X$-computable approximation of $A$.
At stage $s$, we assume that we have already enumerate (a finite fragment of) a copy of $\prec_{\langle a,A_s(a)\rangle}$, and infinitely many copies of $\prec_\omega$.
At stage $s+1$, if $A_{s}(a)<A_{s+1}(a)$ for some $a<s$, the original copy of $\prec_{\langle a,A_s(a)\rangle}$ can be modified to a copy $\prec_{\langle a,A_{s+1}(a)\rangle}$ by enumerating the rationals at the top.
If $A_{s+1}(a)<A_{s}(a)$, the original copy of $\prec_{\langle a,A_s(a)\rangle}$ can be modified to a copy of $\prec_\omega$ by enumerating infinitely many copies of $\mathbb{Q}$ at the top.
Then, start making a new copy of $\prec_{\langle a,A_{s+1}(a)\rangle}$ using fresh elements.
In any case, we also start making a copy of $\prec_{\langle s,A_{s+1}(s)\rangle}$.
One can easily see that this construction gives an $X$-computable copy of $\prec_A$.
\end{proof}

Let us mention a few conclusions of Theorem \ref{thm:degree-spectra-interpolable}.
Recall that $X$ is {\em high} if $\emptyset''\leq_TX'$, and that $X$ is {\em low$_n$} if $X^{(n)}\leq_T\emptyset^{(n)}$.
The {\em degree spectrum} of a relation $\prec$ is the collection of all Turing degrees of isomorphic copies of $\prec$ on $\om$

\begin{corollary}
The class of high degrees is the degree spectrum of an irreflexive interpolable relation.
\end{corollary}

\begin{proof}
Let $A$ in Theorem \ref{thm:degree-spectra-interpolable} be $\emptyset''$.
\end{proof}

\begin{corollary}
There exists a low$_2$ irreflexive interpolable relation which is not isomorphic to any computable relation.
\end{corollary}

\begin{proof}
If $A\not\leq_T\emptyset'$ then $\prec_A$ has no computable isomorphic copy by Theorem \ref{thm:degree-spectra-interpolable}.
Thus, if $X$ is not low, then $\prec_{X'}$ has no computable isomorphic copy, while it has an $X$-computable copy, again by Theorem \ref{thm:degree-spectra-interpolable}.
Now, consider a low$_2$ set $X\leq_T\emptyset'$ which is not low. 
\end{proof}

However, let us recall that effective quasi-Polish spaces are directly correspond to c.e.~transitive relations.
Therefore, when focusing on the topological aspect, it seems more natural to consider a c.e.~relation rather than a computable relation.

\begin{theorem}\label{thm:degree-spectra-interpolable2}
For any $A\subseteq\N$, there exists an irreflexive interpolable relation $\prec_A^\ast$ on $\N$ such that for any $X\subseteq\N$:
\[A\leq_TX''\iff\mbox{there exists an $X$-c.e.~isomorphic copy of $\prec_A^\ast$}.\]
\end{theorem}

\begin{proof}
In the following we fix a bijective coding $\langle\cdot,\cdot,\cdot\rangle\colon\om\times 2\times 2\to\omega$, say $\langle a,b,c\rangle=4a+2b+c$.
Let $\prec_A^\ast$ be the disjoint sum of infinitely many copies of $\prec_{\langle a,i,1\rangle}$ for each $i<2$, and $\prec_{\langle a,A(a),0\rangle}$.
Clearly, $\prec_A^\ast$ is irreflexive and interpolable.

Assume that there exists an $X$-c.e.~relation $\prec$ on $\N$ which is isomorphic to $\prec_A^\ast$.
In this case, for a finite sequence, compatibility is c.e., being an antichain is co-c.e., and maximality is $\Pi^0_2$, relative to $X$.
Hence there exists an $X''$-computable sequence of all finite maximal compatible antichains.
As in the proof of Theorem \ref{thm:degree-spectra-interpolable}, one can see that for any $a\in \N$, $A(a)=b$ if and only if there exists a maximal compatible antichain of size $\langle a,b,0\rangle$.
Thus, one can show that $A\leq_TX''$ as before.

Conversely, if $A\leq_TX''$ then, using their $\Sigma^0_3(X)$-definitions, $A$ and its complement can be presented as follows:
\begin{align*}
a\in A\iff (\exists t\in\N)\liminf_{t\to\infty} A_t(a)=0,\\
a\not\in A\iff (\exists t\in\N)\liminf_{t\to\infty} \bar{A}_t(a)=0,
\end{align*}
where $(a,t)\mapsto A_t(a)$ and $(a,t)\mapsto\bar{A}_t(a)$ are two-valued $X$-computable functions.

In the following, when we refer to a {\em layer} in an isomorphic copy of $\prec_\ell$, we mean the isomorphic image of $\{m\}\times\mathbb{Q}$ for some $m<\ell$.
At each stage $s$, we construct (finite approximations of) infinitely many copies of $\prec_{\langle a,1,A_s(a)\rangle}$ at the $\langle a,1\rangle$-th region, and infinitely many copies of $\prec_{\langle a,0,\bar{A}_s(a)\rangle}$ at the $\langle a,0\rangle$-th region.
We also make junk components, which are infinitely many copies of $\prec_{\langle a,i,1\rangle}$.
At stage $s+1$, if $A_s(a)<A_{s+1}(a)$ then each copy of $\prec_{\langle a,1,{A}_s(a)\rangle}$ can be modified to a copy $\prec_{\langle a,1,{A}_{s+1}(a)\rangle}$ by enumerating fresh rationals at the top.
Even if $A_s(a)>A_{s+1}(a)$, the height of each copy of $\prec_{\langle a,1,{A}_s(a)\rangle}$ can be lowered by adding a relation that moves upper layers of rationals to the right of lower layers of rationals (see the next paragraph for a more formal explanation).
%Each layer of the resulting component is still a discrete liner order.
In this way, each copy of $\prec_{\langle a,1,{A}_s(a)\rangle}$ can be modified to a copy $\prec_{\langle a,1,{A}_{s+1}(a)\rangle}$.
Do the same for $\bar{A}_s$.

Let us give a more detailed explanation of the lowering construction in the case $A_s(a)>A_{s+1}(a)$.
Each element $x$ in the copy of $\prec_{\langle a,1,A_s(a)\rangle}$ in this region is assigned a role $(a,1,i,p)$ for some $i\leq 1$ and $p\in\mathbb{Q}$.
This role changes at most once in the construction.
At stage $s$, at most $s$ elements are enumerated in the copy.
The relationship between each pair of these elements follows directly the $\prec_{\langle a,1,A_s(a)\rangle}$-relation between the assigned roles.
Only at the stage where $A_s(a)>A_{s+1}(a)$ happens, the role of each element $x$ of the $\langle a,1,1\rangle$-th layer changes to play $(a,1,0,p)$ for sufficiently large $p$.
If $p$ is chosen to be larger than all the rationals mentioned so far, this role-change corresponds to placing this element $x$ above all the elements in the lower layers.
Here, only a finite number of elements are already there, and furthermore, an element of a lower layer never comes above an element of a higher layer, so the role-change can be performed consistently.
These are all the possibilities for a role-change to occur, only an element in the $\langle a,1,1\rangle$-th layer can cause a role-change, and such an element drops to the $\langle a,1,0\rangle$-th layer after the role change; hence the role change can occur only once at most.
Finally,  at each stage $s$, for each $n<s$ and $i\leq A_s(a)$, if there is no element with role $(a,1,i,p_n)$ where $p_n$ is the $n$th rational, we arrange a fresh element to give the role of $(a,1,i,p_n)$.

If $A_s(a)=i$ for almost all $s$, then each component at the $\langle a,1\rangle$-th region stabilizes to $\prec_{\langle a,1,i\rangle}$.
Otherwise, we have $A_s(a)>A_{s+1}(a)$ for infinitely many $s$.
Then, if an element $u$ in a component at the $\langle a,1\rangle$-th region belongs to a higher layer at some stage $s$, then at a stage $s'\geq s$ such that $A_{s'}(a)>A_{s'+1}(a)$, adding a relation moves $u$ to a lower layer, after which it never moves from that layer (since the role-change occurs only once at most).
Hence, every element eventually belongs to a lower layer, which means that such a component is isomorphic to $\prec_{\langle a,1,0\rangle}$.
Hence $a\in A$ if and only if there exists a copy of $\prec_{\langle a,1,0\rangle}$.
Indeed, if this is the case, there are infinite such copies.
The same holds for $\bar{A}$.
Hence, we have infinitely many copies of $\prec_{\langle a,A(a),0\rangle}$, and no copies of $\prec_{\langle a,1-A(a),0\rangle}$.
This process may also make an extra copy of $\prec_{a,i,1}$, but in any case, at the junk components, we also make infinitely many copies of $\prec_{a,i,1}$.
Consequently, this construction gives an $X$-computable copy of $\prec_A^\ast$.
\end{proof}

By using this, one can see that there exists a c.e.~irreflexive interpolable relation which is not isomorphic to any computable relation.
Indeed, one can prove a bit stronger.

\begin{corollary}
There exists a c.e.~irreflexive interpolable relation which is not isomorphic to any $X$-computable relation whenever $X$ is not high.
\end{corollary}

\begin{proof}
For $A\subseteq\N$, let $\prec^\ast_A$ be an irreflexive interpolable relation as in Theorem \ref{thm:degree-spectra-interpolable2}.
If $A\leq_TX''$ then $\prec^\ast_A$ has a c.e.~copy.
Moreover, as in the proof of Theorem \ref{thm:degree-spectra-interpolable}, one can easily see that $A\leq_TX'$ if and only if there exists an $X$-computable isomorphic copy of $\prec_A^\ast$.
Hence, for $A=\emptyset''$, if $X$ is not high, i.e., $\emptyset''\not\leq_TX'$, then $\prec^\ast_A$ has no $X$-computable copy.
\end{proof}

\begin{corollary}
There exists a low$_3$-c.e.~irreflexive interpolable relation which is not isomorphic to any c.e.~relation.
\end{corollary}

\begin{proof}
If $A\not\leq_T\emptyset''$ then $\prec_A^\ast$ has no c.e.~isomorphic copy by Theorem \ref{thm:degree-spectra-interpolable}.
Thus, if $X$ is not low$_2$, then $\prec_{X''}^\ast$ has no c.e.~isomorphic copy, while it has an $X$-c.e.~copy, again by Theorem \ref{thm:degree-spectra-interpolable}.
Now, consider a low$_3$ set $X\leq_T\emptyset'$ which is not low$_2$.
\end{proof}

The same result can be obtained for the degree spectra of $\omega$-continuous domains instead of interpolable relations.
\medskip

{\bf Acknowledgement.} 
De Brecht's research was supported by JSPS KAKENHI Grant Number 18K11166.
Kihara's research was partially supported by JSPS KAKENHI Grant Numbers 19K03602 and 21H03392, and the JSPS-RFBR Bilateral Joint Research Project JPJSBP120204809.
Selivanov's research was supported by the RFBR-JSPS Bilateral Joint Research Project 20-51-50001.

\end{document}